\documentclass[dvips,aop,preprint]{article}

\RequirePackage{amsthm,amsmath,amssymb}
\newtheorem{thm}{Theorem}[section]
\newtheorem{theorem}[thm]{Theorem}
\newtheorem{lemma}[thm]{Lemma}
\newtheorem{corollary}[thm]{Corollary}

\newcommand{\Order}{\mbox{O}}

\newcommand{\Real}{\mathbb{R}}
\newcommand{\Int}{\mathbb{Z}}
\newcommand{\Exp}{\mathbb{E}}
\newcommand{\Prob}{\mathbb{P}}
\newcommand{\Pwin}{\mathbb{P}^{(\rm win)}}

\renewcommand{\Vec}[1]{\mbox{\boldmath{$#1$}}}
\usepackage{color}

\usepackage{natbib}

\usepackage[dvipdfmx,colorlinks=true,breaklinks=true,bookmarks=true,urlcolor=blue,
     citecolor=blue,linkcolor=blue,bookmarksopen=false,draft=false]{hyperref}

\title{
Lower Bounds for Bruss' Odds Problem \\ with Multiple Stoppings
}
\author{
Tomomi Matsui (Tokyo Institute of Technology) \\
Katsunori Ano (Shibaura Institute of Technology)
}
\date{}

\begin{document}

\maketitle

%

\begin{abstract}
This paper deals with Bruss' odds problem with multiple stopping chances. 
A decision maker observes sequentially a sequence of independent 0/1 
	(failure/success) random variables with the objective to predict 
	the last success correctly with multiple stopping chances.
First, we give a non-trivial lower bound of the probability of win 
	(obtaining the last success) for the problem with $m$-stoppings. 
Next, we show that the asymptotic value 
	for each classical secretary problem with multiple
	stoppings attains our lower bound. 
Finally, we prove a conjecture on
	the classical secretary problem, 
	which gives a connection between 
	the probability of win 
	and the threshold values of the optimal stopping strategy.
\end{abstract}

\section{Introduction.}

We discuss asymptotic lower bounds  of probability of ``{\em win}'' 
	({\it i.e.}, obtaining the last success) 
	for odds problem with multiple stoppings,
	which has some general setting in optimal stopping theory.
The problem may be stated as follows. 

\begin{quote}
Suppose that you know that 
	you will be given a random sequence of zeros and ones; 
	you know how long the sequence will be, say $N$; 
	you know that you will be given the digits one by one; 
	you do not know what the next digit will be zero or one, 
	but you do know the probability that it will be a one. 
You are allowed to claim at most $m$ times, 
	when you see a one, 
	``the one that I observe now, 
	will be the last one in the sequence.'' 
You will {\em win} if your last claim turns out to be indeed the last one.
\end{quote}

\noindent
The above problem is called an {\em $m$-stopping odds problem}.
You want to win with maximal probability. 
Of course, we assume $m<< N$, 
	for instance $m = 1$, which is the Bruss' optimal stopping problem, 
	and which is closely related to the classical secretary problem.

Now we define a class of feasible policies, say $\Pi$.
Each policy in $\Pi$
	gets input of a sequence of $0/1$ random variables 
	${\cal X}=(X_1, X_2,\ldots , X_N)$,
	that is, Bernoulli sequence.
We say ``{\em success}'' if $X_i=1$ and 
	``{\em failure}'' if $X_i=0$. 
A policy $\pi \in \Pi$ returns (indices of) variables 
	in $\cal X$
	at which you claim to observe the last success.
The value of the policy,
	which is called the {\em probability of win}, 
	is its probability to win:
	$\Prob [\pi (\cal X) =1]$
	(where ``$= 1$'' means ``includes the variable of the last success''). 
An optimal policy would be 
	$\pi^* = \arg \max_{\pi \in \Pi} \Prob [\pi (\cal X) =1 ]$.

This is an attractive problem setting. 
We may quote from Bruss~\cite{BRUSS2000}: 
``Many stopping problems are of a similar kind. 
One often wants to stop on the very last success. 
For instance, investors are typically interested in 
	stopping on the last success in a given period, 
	where a success is a price increase in a long position 
	and a decrease in a short position. 
Similarly, venture capital investors often try to put 
	all reserved capital in the last technological
	innovation in the targeted field. 
In secretary problems, we want to select 
	the best candidate (which means stopping on the last record value)
	and so on.''
	
The single stopping problem has an elegant 
	and a simple optimal stopping strategy 
	determined by {\em Odds theorem\/} or {\em Sum the Odds theorem}. 
A typical lower bound for an asymptotic 
	optimal value (the probability of win), 
	when $N$ goes to infinity, is shown to be $e^{-1}$ 
	in Bruss~\cite{BRUSS2000,BRUSS2003}. 
The value often appears in the literature of the many 
	modifications of the secretary problem having a specified probability 
	of success, $\Prob [X_i=1]=1/i$ 
	(see, e.g., 
	Pfeifer~\cite{PFEIFER1989}, 
	Samuels~\cite{SAMUELS1993} for a review and others), 
	and in the one of the variations of prophet inequality 
	based on relative ranks 
	(see, e.g., Hill and Krengel~\cite{HILLKRENGEL1992}
	and Hill and Kennedy~\cite{HILLKENNEDY1992}).
The value $e^{-1}$ also appears 
	in the asymptotic threshold value of 
	the optimal stopping strategy for the secretary problem. 
For large $N$, 
	the optimal strategy (for the secretary problem)
	is to pass all the candidates until $e^{-1} N$ 
	and then stop at the first relative best (if any) thereafter. 
Other variations of the single stopping problem are studied by 
	Bruss and Paindaveine~\cite{BRUSSPAINDAVEINE2000} 
	for stopping on the $\ell$-th last success,		
	Hsiau and Yang~\cite{HSIAU2002} for Markov-dependent trials,  and 
	Tamaki~\cite{TAMAKI2010} for stopping on any of the last $\ell$ successes.

For the multiple stopping odds problem, 
	Ano, Kakinuma and Miyoshi~\cite{ANO2010}
	provided an optimal strategy
	based on a (multiple) threshold strategy.
%
%
%
%
They also showed that the double stopping odds problem
	has a lower bound $e^{-\frac{3}{2}}+e^{-1}$
	for an asymptotic optimal value (the probability of win)
	when $N$ goes to infinity. 
It is interesting that their lower bound is equal to
	the asymptotic optimal value of the double stopping secretary problem.
Another variation of Markov-dependent trials with multiple stoppings
	is studied by Ano, Kakie~and Miyoshi~\cite{AKM2011}.
	
For odds problems and the secretary problems with multiple stoppings,  
	this paper answers to the following questions.

\smallskip 

\noindent		
(Q1) What is the maximum probability of win and the lower bound 
	of the odds problem (and not only for the secretary problem) 
	with $m$-stoppings?

We give an explicit formula of the probability of win 
	for the $m$-stopping odds problem,
	which is based on an enumeration technique for specified integer sequences.

\smallskip

\noindent
(Q2) What is the asymptotic lower bound of the probability of win 
		for the $m$-stopping odds problem,
		when $N$ goes to infinity? 

For any fixed number $m$ of stopping chances, 
	we give an asymptotic lower bound 
	of the probability of win for odds problem.  
		
\smallskip

\noindent
(Q3) Is the greatest lower bound of the $m$-stopping odds problem
		equal to the asymptotic optimal value (the probability of win) 
		of the secretary problem with $m$-stopping chances?
	In other words, 
		does the secretary problem still keep benchmark position 
		of the bound for odds problem? 

	In the secretary problem, multiple stopping setting may go back to 
		Gilbert and Mosteller~\cite{GM1966}. 
	Asymptotic optimal values of secretary problems 
	for $m=1,2,3$ and $4$ appears in~\cite{GM1966,BRUSS1989}.  	
	The single stopping odds problem has a lower bound $e^{-1}$, 
		shown by Bruss~\cite{BRUSS2003}, 
		which is equal to the asymptotic optimal value 
		of the classical secretary problem.
	For double stopping case ($m=2$), 
		Ano, Kakinuma and Miyoshi~\cite{ANO2010} derived 
		a lower bound $e^{-\frac{3}{2}}+e^{-1}$ of odds problem,
		which is equal to the asymptotic optimal value 
		of the double stopping secretary problem.
	
We prove that the above property holds for any $m$,  
	which also implies the tightness of our lower bounds.
We also give asymptotic optimal values 
	of the secretary problem with $m$ stopping chances
	for $m=5,6,7,8,9$ and $10$ in Table~\ref{tablesumlambda}.
	
\smallskip		
		
\noindent
(Q4) Is the lower bound of the secretary problem
		remarkably composed of 
		the asymptotic threshold values 
		in the optimal multiple stopping strategy?

For example, 
	it is known that the triple of threshold values 
	$(i_N^{(3)}, i_N^{(2)}, i_N^{(1)})$ 
	for the secretary problem 
	with three stopping chances satisfies
\[
	\lim_{N \rightarrow \infty} 
	 \left( \frac{i_N^{(3)}}{N}, \frac{i_N^{(2)}}{N}, \frac{i_N^{(1)}}{N} \right)
	= (e^{-\frac{47}{24}}, e^{-\frac{3}{2}}, e^{-1}), 
\]
and the corresponding probability of win converges to  
	$(e^{-\frac{47}{24}}+e^{-\frac{3}{2}}+e^{-1})$.
It is also known that a similar relation holds 
	when the number of stopping chances is $1,2,3,$ or $4$.
We prove a conjecture on the secretary problem 
	indicated by Gilbert and Mosteller~\cite{GM1966}
	and raised explicitly in~\cite{ANO2010}, 
	which shows a beautiful connection between threshold values 
    of the optimal stopping strategy and the probability of win. 	
Let  $(i_N^{(m)},$ $i_N^{(m-1)},$ $\ldots, i_N^{(1)})$ 
	be a vector of optimal threshold values for the secretary problem
	with $m$-stopping chances,
	where $N$ denotes the number of random variables.
We prove the following equality
\[
	\lim_{N \rightarrow \infty}
	\left( 
		\frac{i_N^{(m)}}{N}+\frac{i_N^{(m-1)}}{N}+\cdots + \frac{i_N^{(1)}}{N} 
	\right)
	= \lim_{N \rightarrow \infty} \Pwin_N (m)
\]
where $\Pwin_N (m)$ denotes the corresponding probability of win 
	for the secretary problem with $m$-stopping chances.

\section{Threshold Strategy and Probability of Win.}\label{SectionTS}

Let ${\cal X}=(X_1, X_2,\ldots , X_N)$ be 
	a given Bernoulli sequence, i.e.,  
	a sequence of independent $0/1$ random variables.
For any $i \in \{1,2,\ldots,N\}$, 
	we denote $ p_i= \Exp[X_i] $.
Throughout this paper, we assume that $0<p_i<1$
	for any $i$.
We denote a probability of failure $1 - p_i$ by $q_i$ and 
	an odds $\displaystyle \frac{p_i}{q_i}$ of $X_i$ by $r_i$,
	for each $i \in \{1,2,\ldots , N\}$.
In this section, we discuss an optimal strategy 
	for the $m$-stopping odds problem defined on  ${\cal X}$. 



%



By applying dynamic programming (DP) techniques,
	Ano, Kakinuma and Miyoshi~\cite{ANO2010} showed that
	an optimal policy has the structure of a multiple threshold strategy, 
	denoted by {\tt Threshold$(i^{(m)},i^{(m-1)},\ldots,i^{(1)})$}, 
	defined by threshold values satisfying 
	$1 \leq i^{(m)} \leq i^{(m-1)} 
	\leq \cdots \leq i^{(1)} \leq N$.
The threshold strategy 
	{\tt Threshold$(i^{(m)},i^{(m-1)},\ldots,i^{(1)})$}
	selects a variable of success  
	(at which you claim to observe the last success) 	
	if and only if the number of previously selected variables
	is less than the number of passed threshold values 
	on and before the observation.
For example, we consider a case that $(X_1,X_2,\ldots, X_8)$
	has a vector of the realized values 
	$(0,1,1,0,0,1,1,1)$.
If we apply the threshold strategy {\tt Threshold$(2,4,5)$}, 
	for the 3-stopping odds problem, 
	then the strategy selects variables $X_2, X_6,X_7$,
	which does not includes the last success $X_8$.
More precisely, the threshold strategy 
	{\tt Threshold$(i^{(m)},i^{(m-1)},\ldots,i^{(1)})$}
	selects a set of variables indexed by a set of indices $\pi (1)$ defined 
	by the following recurrence relation
\begin{eqnarray*}
	\pi (m+1)&=& \emptyset, \\
	\pi (k)&=& \pi (k+1) \cup 
		\left\{ 
			\min 
			\left\{ i \in \{1,2,\ldots , N\} 
				\left|
					\begin{array}{l}
					i^{(k)} \leq i, \;\; X_i=1,  \\
					i' < i \; (\forall i' \in \pi(k+1)).					
					\end{array}
				\right.
			\right\}
		\right\} (\forall k\in  \{m,m-1,\ldots, 1\}),
\end{eqnarray*}
where we define $\{ \min \{\emptyset\} \}=\emptyset$.

First, we discuss the probability of win of 
	{\tt Threshold$(i^{(m)},i^{(m-1)},\ldots,i^{(1)})$}.
We introduce a partition $\{B_{m+1}, B_m, \ldots , B_1\}$ 
	of index set $\{1,2,\ldots,N\}$ 
defined by 
\[
B_k= \left\{ \begin{array}{ll}
	\{i \in \{1,2,\ldots,N \} \mid i^{(1)} \leq i \leq N \} 
			& (k=1), \\
	\{i \in \{1,2,\ldots,N \} \mid i^{(k)} \leq i < i^{(k-1)} \} 
			&(1<k \leq m), \\
	\{i \in \{1,2,\ldots,N \} \mid 1 \leq i<i^{(m)} \}
			&(k=m+1).
			\end{array} \right.
\]

\noindent
Each index set in $\{B_{m+1}, B_m, \ldots , B_1\}$ is called a {\em block}.
We identify the set of indices in each block
	with the set of corresponding variables, 
	if there is no ambiguity.
Given a 0-1 vector $\Vec{x} \in \{0,1\}^N$, 
	we introduce a {\em pattern vector\/} of $\Vec{x}$, 
	denoted by 
	$\Vec{b}(\Vec{x})=(b_m,b_{m-1},\ldots,b_1)$, 
	satisfying that 
	$b_k$ is the number of 1s in the subvector of $\Vec{x}$ 
	defined by block $B_k$, 
	i.e., $b_k=\sum_{i \in B_k}x_i \; (k \in \{1,2,\ldots,m\})$.
Here we note that elements of vector $\Vec{b}(\Vec{x})$
	are arranged in decreasing order of indices 
	and block $B_{m+1}$ is ignored. 
For any vector $\Vec{b}=(b_d,b_{d-1},\ldots,b_1)$, 
	we say that a vector $(b_{d'},b_{d'-1},\ldots,b_1)$ satisfying $d\geq d' \geq 1$
	is a {\em left truncated subvector\/} of $\Vec{b}$.

If an index corresponding to the last success
	is obtained by executing {\tt Threshold}$(\Vec{i})$,
	we say that a vector of realized values
	$(x_1,x_2,\ldots, x_N) \in \{0,1\}^N$
	of $(X_1,X_2,\ldots,X_N)$ is {\em winning}.
When we consider a single-stopping odds problem
	(discussed in~\cite{BRUSS2003}), 
	a vector $\Vec{x} \in \{0,1\}^N$ is winning 
	if and only if its pattern vector $\Vec{b}(\Vec{x})$ satisfies $b_1=1$,
	i.e.,  one-dimensional vector $(1) \in \{0,1\}^1$ 
	is a left truncated subvector of $\Vec{b}(\Vec{x})$.
The probability of win of 
	{\tt Threshold$(i^{(1)})$} is equal to 
$
	\left( \prod_{i\in B_1} q_i \right)\cdot
	\left( \sum_{i \in B_1} r_i \right).
$


Next, we consider the case that $m=2$
	(discussed in~\cite{ANO2010}).
We assume that  $|B_1| \geq 2$.
It is easy to show that 
	a vector $\Vec{x} \in \{0,1\}^N$ is winning
	if and only if its pattern vector $\Vec{b}(\Vec{x})$
	has a left truncated subvector in 
	$\{(1), (1,0), (0,2)\}$.
Thus, the probability of win of {\tt Threshold}$(i^{(2)},i^{(1)})$
	is equal to 
\[
	\left( \prod_{i \in B_1} q_i \right)
	\left( \sum_{i \in B_1} r_i \right)
+	\left( \prod_{i \in B_2 \cup B_1} q_i \right)
	\left( \sum_{i \in B_2} r_i 
+ 	\sum_{\{i,i'\} \subseteq B_1} r_i r_{i'} \right).
\]

When $m \leq 5$, 
	a brute force method shows that 
	a vector $\Vec{x} \in \{0,1\}^N$ is winning 
	if and only if 
	there exists a set 
	$\Xi_k$ $(k\in \{1,2,\ldots ,5\})$ 
	that includes a left truncated subvector of $\Vec{b}(\Vec{x})$, 
	where
\begin{eqnarray}
\Xi_1 &=& \{ (1) \}, 								\label{XI1} \\
\Xi_2 &=& \{ (1,0), (0,2) \}, 						\label{XI2} \\
\Xi_3 &=& \{ (1,0,0), (0,2,0), (0,1,2), (0,0,3) \},	\label{XI3} \\
\Xi_4 &=& \left\{ \begin{array}{l}
		(1,  0,  0,  0),  
		(0,  2,  0,  0),  
		(0,  1,  2,  0), 
		(0,  1,  1,  2), 
		(0,  1,  0,  3), \\
		(0,  0,  3,  0), 
		(0,  0,  2,  2), 
		(0,  0,  1,  3), 
		(0,  0,  0,  4) 
		\end{array} \right\}, 						\label{XI4} \\
\Xi_5 &=& \left\{ \begin{array}{l}
(	1	,	0	,	0	,	0	,	0	)
(	0	,	2	,	0	,	0	,	0	)
(	0	,	1	,	2	,	0	,	0	)
(	0	,	1	,	1	,	2	,	0	)
(	0	,	1	,	1	,	1	,	2	)
(	0	,	1	,	1	,	0	,	3	)\\
(	0	,	1	,	0	,	3	,	0	)
(	0	,	1	,	0	,	2	,	2	)
(	0	,	1	,	0	,	1	,	3	)
(	0	,	1	,	0	,	0	,	4	)
(	0	,	0	,	3	,	0	,	0	)
(	0	,	0	,	2	,	2	,	0	)\\
(	0	,	0	,	2	,	1	,	2	)
(	0	,	0	,	2	,	0	,	3	)
(	0	,	0	,	1	,	3	,	0	)
(	0	,	0	,	1	,	2	,	2	)
(	0	,	0	,	1	,	1	,	3	)
(	0	,	0	,	1	,	0	,	4	)\\
(	0	,	0	,	0	,	4	,	0	)
(	0	,	0	,	0	,	3	,	2	)
(	0	,	0	,	0	,	2	,	3	)
(	0	,	0	,	0	,	1	,	4	)
(	0	,	0	,	0	,	0	,	5	)
		\end{array} \right\}.		
\end{eqnarray}

\noindent
The following table of the sizes of $\Xi_k$ is obtained
	by a naive computer program for enumeration.

\[
\begin{array}{c||r|r|r|r|r|r|r|r|r|r|r|r}
k  & 1 & 2 & 3 & 4 & 5 & 6 & 7 & 8 & 9 & 10 & 11 & \cdots \\
\hline
|\Xi_k| & 1& 2& 4& 9& 23& 65& 197& 626& 2056& 6918& 23714 & \cdots  \\
\end{array}
\]
Here we discuss an example of the 5-stopping odds problem.
Assume that  $\Vec{x} \in \{0,1\}^N$ has a pattern vector
	$\Vec{b}(\Vec{x})=(3,1,0,2,0)$.
The threshold strategy 
	selects first variable of success in block $B_5$,
	rejects second and third variables of success in $B_5$, 
	and selects all the remained three variables of success in $B_4 \cup B_2$.
The vector $\Vec{x}$ is winning, since $\Vec{b}(\Vec{x})=(3,1,0,2,0)$
	has a left-truncated subvector $(0,2,0) \in \Xi_3$.

Now we discuss the general case.
First, we show a necessary and sufficient condition that
	a vector $\Vec{x} \in \{0,1\}^N$ becomes a winning vector
	of the $m$-stopping odds problem.
Throughout this paper, $\Int_+$ denotes the set of non-negative integers.
We define a set of $k$-vectors
\[
	{\Xi}_k =
	\left\{ (b_k,b_{k-1},\ldots ,b_1) \in \Int_+^k 
	\left|
		\begin{array}{l}
			\exists k^* \in \{k, k-1,\ldots ,2,1 \}, \\
			\quad
			\begin{array}{rcl}
				1 & > & b_k, \\
				2 & > & b_k + b_{k-1}, \\
				  &  \vdots & \\
				k-k^*   & > & b_k +b_{k-1} + \cdots +b_{k^*+1}, \\
		        1+k-k^* & = & b_k +b_{k-1} + \cdots +b_{k^*+1} +b_{k^*}, \\			
			\end{array} \\
			\quad b_{k^*-1}= b_{k^*-2}= \cdots =b_1=0.
		\end{array}
	\right\} 
	\right.
\]
	for each positive integer $k$.
 Here we note that when $k^*=k$ in the above definition, 
 	all the inequalities are ignored and 
 	$(1,0,0,\ldots,0) \in \Int_+^k $ becomes a unique vector 
 	satisfying the remaining conditions.
Thus,  the set $\Xi_k$ contains the $k$-vector $(1,0,0,\ldots,0) $.
%

\begin{lemma}
Let $\Vec{x} \in \{0,1\}^N$ be a vector satisfying that
	$\exists k \in \{1,2,\ldots,  m\}$, 
	$\exists \Vec{b} \in \Xi_k$,
	$\Vec{b}$ is a left truncated subvector of $\Vec{b}(\Vec{x})$.
Then $\Vec{x}$ is a winning vector of the $m$-stopping odds problem.
\end{lemma}

\mbox{Proof. }
Let $k^*$ be an index of variable appearing in the definition of $\Xi_k$
	with respect to $\Vec{b} \in \Xi_k$.
The definition of $\Xi_k$ implies that 
	block $B_{k^*}$ includes 
	the index of the variable of the last success
	and the number of successes with respect to $\Vec{x}$ in 
	blocks $B_k \cup B_{k-1} \cup \cdots \cup B_{k^*}$
	is equal to $1+k-k^*$.

When we apply the threshold strategy,  
	the number of selected variables
	in blocks  $B_m \cup B_{m-1} \cup \cdots \cup B_{k+1}$ 
	is less than or equal to $m-k$, obviously.

Thus, when we observe the variables of the last success,
	the number of previously selected variables
	is less than or equal to $(m-k)+(1+k-k^*)-1=m-k^*$,
	which is less than the number of passed threshold values $m-k^*+1$, 
	where passed threshold values are $\{i^{(m)}, \ldots, i^{(k^*)}\}$.

 It implies that the threshold strategy selects 
	the variable of the last success and thus $\Vec{x}$ is winning.
(Here we note that all the variables of success 
	in $B_k \cup B_{k-1} \cup \cdots \cup B_{k^*}$
	are also selected.)
\hfill \qed 

Next, we discuss the inverse implication.

\begin{lemma}
For any winning vector $\Vec{x} \in \{0,1\}^N$ 
	of the $m$-stopping odds problem, 
	there exists an index $k \in \{1,2,\ldots,  m\}$
	and a vector $\Vec{b} \in \Xi_k$ satisfying that 
	$\Vec{b}$ is a left truncated subvector of $\Vec{b}(\Vec{x})$.
\end{lemma}

\mbox{Proof. }
We show the above by induction on $m$.
When $m=1$,  it is obvious.
We assume that the above property holds for each integer in $\{1,2,\ldots, m-1\}$.

\noindent
(i) Consider a case that
	when we apply the threshold strategy 
	{\tt Threshold$(i^{(m)},\ldots,i^{(1)})$}, 
	the number of previously selected variables is 
	strictly less than the number of passed threshold values 
	at every time instance just after a variable is selected.
Then, it is easy to see that 
	even if we apply the threshold strategy 
	{\tt Threshold$(i^{(m-1)},\ldots,i^{(1)})$} to $\Vec{x}$, 
	the set of selected variables remains unchanged, 
	and thus $\Vec{x}$ is also a winning vector 
	of the $(m-1)$-stopping odds problem.
The induction hypothesis implies that, 
	$\exists k \in \{1,2,\ldots,  m-1\}$, 
	$\exists \Vec{b} \in \Xi_k$ satisfying that
	$\Vec{b}$ is a left truncated subvector of $\Vec{b}(\Vec{x})$.

\noindent
(ii) Next, we consider the remaining case.
Let $i^* \in \{1,2,\ldots , N\}$ be the minimum index
	of a variable selected by the threshold strategy 
	{\tt Threshold$(i^{(m)},\ldots,i^{(1)})$} 
	under the condition that 
	just after $x_{i^*}$ is selected, 
	the number of previously selected variables becomes  
	equal to the number of passed threshold values. 
We denote the block including $i^*$ by $B_{k^*}$.
From the minimality of $i^*$, 
	$m$-vector $\Vec{b}(\Vec{x})=(b_m,b_{m-1},\ldots ,b_1)$ 
	satisfies inequalities
\begin{eqnarray*}
	1 & > & b_m, \\
	2 & > & b_m + b_{m-1}, \\
	&  \vdots & \\
	m-k^*   & > & b_m +b_{m-1} + \cdots +b_{k^*+1}.
\end{eqnarray*} 
Here we note that when $k^*=m$, 
	we omit the above inequalities.
Since the number of previously selected variables is 
	equal to the number of passed threshold values 
	just after $x_{i^*}$ is selected, 
	all the variables 
	in $\{X_i \mid i \in B_{k^*} \mbox{ and } i^*<i\}$ 
	are not selected and thus we have the equality 
\begin{eqnarray} \label{first-apex}
	1+m-k^*  =  b_m +b_{m-1} + \cdots +b_{k^*+1} +b_{k^*}.
\end{eqnarray}

\noindent 
(ii-a) If $x_{i^*}$ is the last success, 
	then $b_{k^*-1}= b_{k^*-2}= \cdots =b_1=0$, 
	which yields $\Vec{b}(\Vec{x}) \in \Xi_m$.

\noindent
(ii-b) 
Lastly, we consider the case that $x_{i^*}$ is not the last success.
Clearly, $B_{k^*}$ does not include 
	the index of the variable of the last success, 
	since variables 
	in $\{x_i \mid i \in B_{k^*} \mbox{ and } i^*<i\}$ 
	are not selected. 

The equality~(\ref{first-apex}) implies that 
	if we apply 
	{\tt Threshold$(i^{(k^*-1)},\ldots,i^{(1)})$} to $\Vec{x}$,  
	the set of selected variables in $B_{k^*-1} \cup \cdots \cup B_2 \cup B_1$
	remains unchanged and thus 
	$\Vec{x}$ is a winning vector 
	of the $(k^*-1)$-stopping odds problem.  
The induction hypothesis implies the desired result. 
\hfill \qed



We can also show the following uniqueness.

\begin{corollary}
For any winning vector $\Vec{x} \in \{0,1\}^N$ 
	of the $m$-stopping odds problem, 
	there exists a unique index $k \in \{1,2,\ldots,  m\}$
	and a unique vector $\Vec{b} \in \Xi_k$ satisfying that  
	$\Vec{b}$ is a left truncated subvector of $\Vec{b}(\Vec{x})$.
\end{corollary}

\mbox{Proof. }
The definition of $\Xi_k$ directly implies that 
	for any pair of vectors 
	$\Vec{b'} \in \Xi_{k'}$ and $\Vec{b''} \in \Xi_{k''}$, 
	when $\Vec{b''}$ is a left truncated subvector of $\Vec{b'}$, 
	then both $\Vec{b'}=\Vec{b''}$ and $k'=k''$ hold.
As a consequence, we have the desired result.
\hfill \qed

Summarizing the above properties, we have the following theorem.

\begin{theorem}\label{winningIFF}
A vector $\Vec{x} \in \{0,1\}^N$ is a winning vector
	of the $m$-stopping odds problem if and only if
	there exists a unique integer $k \in \{1,2,\ldots, m\}$
	satisfying that
	a pattern vector $\Vec{b}(\Vec{x})=(b_m,b_{m-1},\ldots,b_1)$
	has a left truncated subvector 
	$(b_k,b_{k-1},\ldots ,b_1) \in \Xi_k$.
\end{theorem}	
	
Next, we discuss the probability of win.
Given an index subset $B \subseteq \{1,2,\ldots ,N\}$
	and a positive integer $b$,
	we define 
\begin{eqnarray*}
f^b(B) &=& \left\{ \begin{array}{cl}
	\displaystyle 
	\sum_{B' \subseteq B, \; |B'|=b} 
	\left( \prod_{i \in B'} r_i \right) & (|B|\geq b), \\
	0									& (|B|< b), 
		\end{array} \right.  \\
\end{eqnarray*}
which is called an {\em elementary symmetric polynomial} defined on 
	$\{r_i \mid i \in B\}$.
We define $f^0(B) = 1$.


\begin{theorem}
Given a threshold strategy 
	${\tt Threshold}(i^{(m)}, \ldots, i^{(2)},i^{(1)})$
	for the $m$-stopping odds problem
	defined on 
	$X_1,X_2,\ldots,X_N$,
	the corresponding probability of win is equal to
\[
	\sum_{k=1}^m 
	\left(
	\left( \prod_{i \in B_k \cup \cdots \cup B_2 \cup B_1 } q_i \right)
	\sum_{(b_k,\ldots, b_1) \in \Xi_k}
		\left(
			f^{b_k}(B_k) f^{b_{k-1}}(B_{k-1}) \cdots f^{b_1}(B_1)
		\right)
	\right).
\]
\end{theorem}

 \mbox{Proof. }
The definition of $f^b (B)$ directly implies that 
	for any $(b_k,b_{k-1},\ldots,b_1)\in \Int_+^k$, 	
\begin{eqnarray*}
	\Prob \left[ 
		\bigcap_{k'=1}^{k}  
		\left\{ \sum_{i \in B_{k'}}X_i = b_{k'} \right\}
	\right]
	&=& 
	\left( \left( \prod_{i \in B_k}q_i 		\right) f^{b_k}(B_k) \right)
	\cdots
	\left( \left( \prod_{i \in B_1}q_i 		\right) f^{b_1}(B_1) \right) \\
	&=& 
	\left( \prod_{i \in B_k \cup \cdots \cup B_2 \cup B_1 } q_i \right)
	\left(
		f^{b_k}(B_k) f^{b_{k-1}}(B_{k-1}) \cdots f^{b_1}(B_1)
	\right).
\end{eqnarray*}
From the uniqueness appearing in Theorem~\ref{winningIFF}, 
	the probability of win of the threshold strategy 
	${\tt Threshold}(i^{(m)}, \ldots, i^{(2)},i^{(1)})$
	is equal to
\begin{eqnarray*}
\lefteqn{
	\sum_{k=1}^m \sum_{(b_k,\ldots, b_1) \in \Xi_k}
	\Prob \left[ 
			\bigcap_{k'=1}^{k}  
		\left\{ \sum_{i \in B_{k'}}X_i = b_{k'} \right\}
	\right]
}\\
	&=&
	\sum_{k=1}^m 
	\left(
	\left( \prod_{i \in B_k \cup \cdots \cup B_2 \cup B_1 } q_i \right)
	\sum_{(b_k,\ldots, b_1) \in \Xi_k}
		\left(
			f^{b_k}(B_k) f^{b_{k-1}}(B_{k-1}) \cdots f^{b_1}(B_1)
		\right)
	\right). \\[-5ex]
\end{eqnarray*}
\hfill \qed
\smallskip

\noindent
For example, when $m=3$, 
	the set $\Xi_3$ of winning patterns 
	includes four vectors
	$\Xi_3=\{(1,0,0), (0,2,0), (0,1,2), (0,0,3)\}$
	and thus the probability of win is equal to
\[
\begin{array}{l}
\displaystyle
	\left( \prod_{i \in B_1} q_i \right)
	\left( \sum_{i \in B_1} r_i \right)
+	\left( \prod_{i \in B_2 \cup B_1} q_i \right)
	\left( \sum_{i \in B_2} r_i 
+ 	\sum_{\{i_1,i_2\} \subseteq B_1} r_{i_1} r_{i_2} \right) \\
\displaystyle
+	\left( \prod_{i \in B_3 \cup B_2 \cup B_1} q_i \right)
	\left(
	\begin{array}{l}
	\displaystyle
		\sum_{i \in B_3} r_i 
	+ 	\sum_{\{i_1,i_2 \} \subseteq B_2} r_{i_1} r_{i_2} \\
	\displaystyle
	+	\left( \sum_{i \in B_2}r_i \right)
		\left(\sum_{\{i_1,i_2\} \subseteq B_1} r_{i_1}r_{i_2}\right)
	+	\sum_{\{i_1,i_2,i_3\} \subseteq B_1} r_{i_1}r_{i_2}r_{i_3}
	\end{array}
	\right). \\
\end{array}
\]

\section{Lower Bounds.}

In this section, 
	we discuss a lower bound of the probability of win.


First, we discuss a solution vector 
	$(\lambda_1, \lambda_2, \ldots, \lambda_m)$ 
	of an equality system;
\begin{equation} \label{EQSystem}
\sum_{(b_k,\ldots ,b_1) \in \Xi_k}
	\left( 
	\frac{\lambda_{k  }^{b_k}    }{b_k!    } 
	\frac{\lambda_{k-1}^{b_{k-1}}}{b_{k-1}!} 
	\cdots 
	\frac{\lambda_{1  }^{b_1    }}{b_1!} \right)=1
	\;\;\;\;\; (k\in \{1,2,\ldots, m\}),
\end{equation}
which plays an important role in this paper.
Let us begin with small examples.
The description of $ \Xi_1, \Xi_2, \Xi_3$ and $\Xi_4$ 
	(see (\ref{XI1})--(\ref{XI4})) implies that 
	$(\lambda_1, \lambda_2, \lambda_3, \lambda_4)$
	is a solution of the following system;	
\begin{eqnarray*}
  \frac{\lambda_1^1}{1!} &=& 1, \\
  \frac{\lambda_2^1 \lambda_1^0}{1!\; 0!}
+ \frac{\lambda_2^0 \lambda_1^2}{0!\; 2!} &=& 1, \\
  \frac{\lambda_3^1 \lambda_2^0 \lambda_1^0}{1!\; 0!\; 0!}
+ \frac{\lambda_3^0 \lambda_2^2 \lambda_1^0}{0!\; 2!\; 0!}
+ \frac{\lambda_3^0 \lambda_2^1 \lambda_1^2}{0!\; 1!\; 2!}
+ \frac{\lambda_3^0 \lambda_2^0 \lambda_1^3}{0!\; 0!\; 3!} &=& 1, \\[2ex]
  \frac{\lambda_4^1 \lambda_3^0 \lambda_2^0 \lambda_1^0}{1!\; 0!\; 0!\; 0!}
+ \frac{\lambda_4^0 \lambda_3^2 \lambda_2^0 \lambda_1^0}{0!\; 2!\; 0!\; 0!}
+ \frac{\lambda_4^0 \lambda_3^1 \lambda_2^2 \lambda_1^0}{0!\; 1!\; 2!\; 0!}
+ \frac{\lambda_4^0 \lambda_3^1 \lambda_2^1 \lambda_1^2}{0!\; 1!\; 1!\; 2!}
+ \frac{\lambda_4^0 \lambda_3^1 \lambda_2^0 \lambda_1^3}{0!\; 1!\; 0!\; 3!} \qquad \\
+ \frac{\lambda_4^0 \lambda_3^0 \lambda_2^3 \lambda_1^0}{0!\; 0!\; 3!\; 0!}
+ \frac{\lambda_4^0 \lambda_3^0 \lambda_2^2 \lambda_1^2}{0!\; 0!\; 2!\; 2!}
+ \frac{\lambda_4^0 \lambda_3^0 \lambda_2^1 \lambda_1^3}{0!\; 0!\; 1!\; 3!}
+ \frac{\lambda_4^0 \lambda_3^0 \lambda_2^0 \lambda_1^4}{0!\; 0!\; 0!\; 4!} &=& 1.
\end{eqnarray*}

\noindent
The above system has a solution 
	$(\lambda_1, \lambda_2, \lambda_3, \lambda_4)
	=(1, 1/2, 11/24, 505/1152)$.
(Our theorem appearing below shows that $e^{-1}+e^{-3/2}+e^{-47/24}+e^{-2761/1152}$ 
	gives a lower bound of the probability of win.)
Here we note that 
	$\lambda_5=209519/491520$ and 
	$\lambda_6=49081919440723/117413668454400$.
Values $\lambda_7, \ldots , \lambda_{10}$ 
	appear in our mimeo~\cite{MATSUI2012}.
First, we show the uniqueness of a solution.
The following property is discussed 
	by Gilbert and Mosteller~\cite{GM1966}
	in a setting of the secretary problem.
	
\begin{lemma}\label{unitvector} 
For any $k \in \{1,2,\ldots, m\}$,
	$\Xi_k$ includes the unit $k$-vector $\Vec{e}=(1,0,0,\ldots,0)$ and
	every vector $(b_k,b_{k-1},\ldots ,b_1) \in \Xi_k \setminus \{\Vec{e}\}$
	satisfies $b_k=0$.
\end{lemma}	

 \mbox{Proof. }
It is obvious that $\Vec{e} \in \Xi_k$.
Conversely, if a vector 
	$\Vec{b}=(b_k, \ldots ,b_1) \in \Xi_k$
	satisfies $b_k \neq 0$, 
	then the index $k^*$ appearing in the definition of $\Xi_k$
	satisfies $k^*=k$ and thus  $\Vec{b}=\Vec{e}$.
\hfill \qed 
\smallskip

The above lemma says that 
	the $k$-th equality of~(\ref{EQSystem}) includes 
	$k$ variables $\lambda_k, \lambda_{k-1}, \ldots, \lambda_1$ and
	is a linear equality with respect to $\lambda_k$.
Thus, the equality system~(\ref{EQSystem}) has a unique solution.

\begin{theorem}\label{PositiveSol}
	Equality system~\mbox{\rm (\ref{EQSystem})}
	has a unique solution $(\lambda_1, \lambda_2, \ldots, \lambda_m)$
	satisfying $\lambda_k>0$ for every $k\in \{1,2,\ldots, m\}$.
\end{theorem}

A proof is given in appendix section.

\medskip

	
Given a pair of Bernoulli sequences ${\cal X}$ and ${\cal X}'$, 
	we denote ${\cal X} \cdot \hspace{-0.8ex} \succ {\cal X}'$, 
	if and only if 
	$\exists i \in \{1,2,\ldots, N\}$, 
	satisfying 
  ${\cal X}=(X_1, X_2, \ldots , X_i, \ldots , X_N)$,
	${\cal X}'=
   (X_1,X_2,\ldots,X_{i-1}, X'_i, X''_i, X_{i+1}, \ldots, X_N)$,
	and
	$\Prob [X_i=0]=\Prob[X'_i=0] \Prob[X''_i=0]$.
We say that a Bernoulli sequences ${\cal X}''$ is a {\em subdivision}
	of ${\cal X}$, denoted by ${\cal X} \succeq {\cal X}''$, 
	if and only if either ${\cal X}$ is equivalent to ${\cal X}''$ or  
	there exists a finite sequence 
	${\cal X} \cdot \hspace{-0.8ex} \succ  {\cal X}^1 
	\cdot \hspace{-0.8ex} \succ  {\cal X}^2
	\cdot \hspace{-0.8ex} \succ  \cdots 
	\cdot \hspace{-0.8ex} \succ  {\cal X}''$.

Now we consider the $m$-stopping odds problems
	defined on ${\cal X}$ and ${\cal X}'$ satisfying ${\cal X} \succeq {\cal X}'$.
	
\begin{lemma}\label{subdivisionLB}
Let ${\cal X}$ be a Bernoulli sequence and ${\cal X}'$ be a subdivision of ${\cal X}$,
	i.e.,  ${\cal X} \succeq {\cal X}'$.
Odds problems with $m$-stoppings 
	defined on ${\cal X}$ and ${\cal X}'$ satisfies that
	the probability of win of an optimal strategy for ${\cal X}$ 
	is greater than or equal to that of ${\cal X}'$.
\end{lemma}

 \mbox{Proof. }
 Obviously, we only need to consider the case that 
		${\cal X} \cdot \hspace{-0.8ex} \succ  {\cal X}'$.\@
 In the following, we show that there exists a strategy for ${\cal X}$ 
	whose probability of win 
	is equivalent to that of a (fixed) optimal strategy for ${\cal X}'$.
	
 Since ${\cal X} \cdot \hspace{-0.8ex} \succ  {\cal X}'$, 
	there exists an index $i \in \{1,2,\ldots, N\}$, 
	satisfying  ${\cal X}=(X_1, X_2, \ldots X_N)$,
	${\cal X}'=(X_1,X_2,\ldots,X_{i-1}, X'_i, X''_i, X_{i+1}, \ldots, X_N)$,
	and $q=q' q'' $, where 
    $q=\Prob[X_i=0], \; q'=\Prob[X'_i=0], \; q''=\Prob[X''_i=0] $.
	
  When we observe variables 	$(X_1,X_2,\ldots,X_{i-1})$ in ${\cal X}$, 
	we apply the optimal strategy for ${\cal X}'$ and halt just before observing $X_i$.

\noindent 
(1) If $X_i=0$,  then we do not select $X_i$ and put $(X'_i, X''_i)=(0,0)$. 
	We re-start the optimal strategy for ${\cal X}'$
	and apply to observed sequence of random variables  $(X_{i+1}, \ldots , X_N)$.

\noindent
(2)  In case that $X_i=1$, we choose $(X'_i, X''_i) \in \{(1,0), (0,1), (1,1) \}$ 
	according to the following conditional probabilities; 
\[
		(X'_i, X''_i)= \left\{ \begin{array}{llr}
		(1,0)& \mbox{with probability }& (1-q')q''/(1-q'q''), \\
		(0,1)& \mbox{with probability }& q'(1-q'')/(1-q'q''), \\
		(1,1)& \mbox{with probability }& (1-q')(1-q'')/(1-q'q''). \\
							\end{array} \right.
\]
 We apply the optimal strategy for ${\cal X}'$ to the chosen vector of $(X'_i, X''_i)$. 
 If the optimal strategy for ${\cal X}'$ selects 
	the last success in $(X'_i, X''_i)$, 
	we also select $X_i$ in ${\cal X}$  and vise verse.
 We re-start the optimal strategy for ${\cal X}'$
	and apply to observed sequence of random variables  $(X_{i+1}, \ldots , X_N)$.
 
 It is clear that, the strategy for ${\cal X}$, described above, 
	selects the last success of ${\cal X}$, 
	if and only if the optimal strategy for ${\cal X}'$ 
	selects the last success of ${\cal X}'$.
 Thus, the probability of win of the above strategy 
	is equivalent to that of the optimal strategy for ${\cal X}'$.
\hfill \qed



\begin{theorem}\label{mainLB}
Let $(\lambda_1, \lambda_2,\ldots, \lambda_m)$ be a unique solution
	of equality system~\mbox{\rm (\ref{EQSystem})}.
If a Bernoulli sequence
	$X_1,X_2,\ldots, X_N$ satisfies 
$
	\prod_{i=1}^N q_i  < e^{-\sum_{k=1}^m \lambda_k}, 
$
	then the probability of win of an optimal strategy
	for the $m$-stopping odds problem defined on $X_1,X_2,\ldots, X_N$
	is greater than or equal to 
\[ 
	\sum_{k=1}^m e^{-\sum_{k'=1}^k \lambda_{k'}}
	= e^{-\lambda_1}
	+ e^{-(\lambda_1+\lambda_2)}
	+ e^{-(\lambda_1+\lambda_2+\lambda_3)}
	+\cdots 
	+ e^{-(\lambda_1+\cdots +\lambda_m)}.
\]
\end{theorem}

This theorem says that;
	when $m=3,4,5$, our lower bounds are  
	$e^{-1}+e^{-\frac{3}{2}}+e^{-\frac{47}{24}}\geq 0.7321029820$,  
	$e^{-1}+e^{-\frac{3}{2}}+e^{-\frac{47}{24}}+e^{-\frac{2761}{1152}}
		\geq 0.8231206726$, 
	and 
	$e^{-1}+e^{-\frac{3}{2}}+e^{-\frac{47}{24}}+e^{-\frac{2761}{1152}}
		+e^{-\frac{4162637}{1474560}}
		\geq 0.8825499145$, respectively.	
Table~\ref{tablesumlambda} shows our lower bounds
	in cases $m\leq 10$.

\smallskip

 \mbox{Proof. } 
Let ${\cal X}^0=(X_1,X_2,\ldots ,X_N)$ be a given Bernoulli sequence.
We construct a sequences ${\cal X}^1$
	satisfying  ${\cal X}^0 \cdot \hspace{-0.8ex} \succ {\cal X}^1$
	by splitting a variable $X_i$ in ${\cal X}^0$
	which attains the minimum $\min \{q_1,q_2,\ldots , q_N\}$, 
	and introducing two random variables 
	$X'_i, X''_i$ satisfying 
	$\Prob [X'_i=0]=\Prob [X''_i=0] = \sqrt{q_i}$.
Applying the above procedure iteratively, 
	we obtain an infinite sequence 
	${\cal X}^0 	\cdot \hspace{-0.8ex} \succ {\cal X}^1
				\cdot \hspace{-0.8ex} \succ {\cal X}^2
				\cdot \hspace{-0.8ex} \succ \cdots $.
It is clear that the maximum of odds of variables in ${\cal X}^d$ 
	approaches to $0$ as $d \rightarrow \infty$.

Let $({\cal Y}^1, {\cal Y}^2, \ldots )$ 
	be a subsequence of $({\cal X}^1,{\cal X}^2,\ldots )$
	satisfying that the maximum odds of ${\cal Y}^d$ 
	is less than or equal to $2^{-d}$.
We denote the Bernoulli sequence ${\cal Y}^d$ by $(Y^d_1,Y^d_2,\ldots,Y^d_{L_d})$
	and the corresponding odds by $(r^{(d)}_1,r^{(d)}_2,\ldots , r^{(d)}_{L_d})$,
	where  $L_d$ denotes the length of ${\cal Y}^d$.
We introduce a specified threshold strategy 	
	{\tt Threshold}$(j_d^{(m)},j_d^{(m-1)},\ldots,$ $j_d^{(1)})$
	for the $m$-stopping odds problem on $(Y^d_1,Y^d_2,\ldots,Y^d_{L_d})$
	defined by 
\begin{equation}\label{LBthreshold}
	j_d^{(k)}= \min \left\{ j \in \{1,2,\ldots , L_d \} \left| 
		\lambda_1+\lambda_2+ \cdots +\lambda_k > \sum_{j'=j}^{L_d} r^{(d)}_{j'}  
			\right. \right\}, 
\end{equation}
	for each $k \in \{1,2,\ldots ,m\}$.
Theorem~\ref{PositiveSol} implies inequalities
	$1 \leq j_d^{(m)} \leq j_d^{(m-1)} \leq \cdots \leq j_d^{(1)} \leq L_d$
	and thus we can define a corresponding threshold strategy.
Let $\{ B_{m+1}(d),B_m(d),\ldots,B_1(d)\}$ 
	be a partition of index set $\{1,2,\ldots,L_d \}$
	defined by 
\[
B_k(d)= \left\{ \begin{array}{ll}
	\{j \in \{1,2,\ldots,L_d \} \mid j_d^{(1)} \leq j \leq L_d \} 
			& (k=1), \\
	\{j \in \{1,2,\ldots,L_d \} \mid j_d^{(k)} \leq j < j_d^{(k-1)} \} 
			&(1<k \leq m), \\
	\{j \in \{1,2,\ldots,L_d \} \mid 1 \leq j<j_d^{(m)} \}
			&(k=m+1). 
			\end{array} \right.
\]

Here we show that $1< j_{(d)}^m$.
Since ${\cal Y}^d$ is a subdivision of ${\cal X}$, 
	the probabilities of failure 
	$q^{(d)}_i = \Prob [Y^d_i=0]$ and 
	$q_i = \Prob [X_i=0]$
	satisfy  $q^{(d)}_1 q^{(d)}_2 \cdots q^{(d)}_{L_d}=q_1 q_2 \cdots q_N$.
The total sum of odds of ${\cal Y}^d$ satisfies that
\begin{eqnarray*}
	\sum_{i=1}^{L_d} r^{(d)}_i 
	&=&  	\sum_{i=1}^{L_d} \left( \frac{1}{q^{(d)}_i} -1 \right) 
	=   	\sum_{i=1}^{L_d} \left( \frac{1}{q^{(d)}_i} \right) - L_d  
	\geq 	L_d \left( \prod_{i=1}^{L_d} \frac{1}{q^{(d)}_i} \right)^{\frac{1}{L_d}}- L_d 
	=		L_d \left( \frac{1}{\prod_{i=1}^{L_d} q^{(d)}_i} \right)^{\frac{1}{L_d}}- L_d \\
	&=&		L_d \left( \frac{1}{\prod_{i=1}^N q_i} \right)^{\frac{1}{L_d}}- L_d 
	> 		L_d \left( \frac{1}{e^{-\sum_{k=1}^m \lambda_k}} \right)^{\frac{1}{L_d}}- L_d 
	=		L_d \left( \left( e^{\sum_{k=1}^m \lambda_k} \right)^{\frac{1}{L_d}}-1 \right) \\
	&\geq& \lim_{L' \rightarrow \infty} 
		L' \left( \left( e^{\sum_{k=1}^m \lambda_k} \right)^{\frac{1}{L'}}-1 \right) 
	=	\ln \left( e^{\sum_{k=1}^m \lambda_k} \right) = \sum_{k=1}^m \lambda_k.
\end{eqnarray*}
From the above, the total sum of odds of ${\cal Y}^d$ 
	is strictly greater than $ \sum_{k=1}^m \lambda_k$
	and thus $1< j_{(d)}^m$, 
	which implies $1 \in B_{m+1}(d)\neq \emptyset.$
	
\noindent (i)
First, we show that
\begin{equation}\label{LBb1}
	\lim_{d \rightarrow \infty} \sum_{i \in B_k (d)} r^{(d)}_i= \lambda_k  \mbox{ and }
	\lim_{d \rightarrow \infty} | B_k (d) |= + \infty, \;\;\; 
	\forall k \in \{1,2,\ldots, m\}.
\end{equation}
	
Since the odds of ${\cal Y}^d$ is less than $2^{-d}$, 
	block $B_k(d)$ satisfies
\begin{equation}\label{LBb2}
\lambda_k - 2^{-d} \leq \sum_{i \in B_k (d)} r^{(d)}_i \leq \lambda_k+2^{-d}
	\;\; (\forall k \in \{1,2,\ldots, m\})
\end{equation}
and thus, we obtain the first equality in~(\ref{LBb1}).
Since $\displaystyle 
		\lambda_k - 2^{-d} 
		\leq \sum_{i \in B_k(d)} r^{(d)}_i
		\leq |B_k (d)|  2^{-d}$, 
	it is clear that
	$\displaystyle 
		\lim_{d \rightarrow \infty}|B_k (d)| 
		\geq \lim_{d \rightarrow \infty} (2^d \lambda_k -1) = +\infty$.

\noindent (ii)
It is easy to show that for any  $ k \in \{1,2,\ldots, m\}$,
\begin{eqnarray*}
\lim_{d \rightarrow \infty} 
	\prod_{j \in B_k(d)} q^{(d)}_i  
&=& \lim_{d \rightarrow \infty} 
	\prod_{j \in B_k(d)} 
		\left( \frac{1}{1+r^{(d)}_i} \right) 
= 	\lim_{d \rightarrow \infty} 
	\left(
		\prod_{j \in B_k(d)} (1+r^{(d)}_i) 
	\right)^{-1}\\				
&\geq& 
	\lim_{d \rightarrow \infty} 
	\left(
		\frac{
			\sum_{j \in B_k(d)} (1+r^{(d)}_i)
		}{|B_k(d)|}
	\right)^{-|B_k(d)|}
=	\lim_{d \rightarrow \infty} 
	\left(
		1+ \frac{
			\sum_{j \in B_k(d)} r^{(d)}_i
		}{|B_k(d)|}
	\right)^{-|B_k(d)|}		
= e^{-\lambda_k}.
\end{eqnarray*}

\noindent(iii) Third, we show that the elementary symmetric polynomials
		satisfy 
\[ 
	\forall k \in \{1,2,\ldots,m\}, 
	\forall b \in \{0,1,\ldots,m\}, 
		\lim_{d \rightarrow \infty} f^b(B_k(d)) 
		\geq \frac{\lambda_k^b}{b!}.
\] 

When $b=0$, the definition of $f^0(B_k(d))$ says that
	 $f^0(B_k(d))=1=\frac{\lambda_k^0}{0!}$ holds permanently.
Let us consider the cases that $b \geq 1$.
When $d$ is sufficiently large, 
	the size of  $B_k(d)$ is greater than $m$,
	since	$\lim_{d \rightarrow \infty} |B_k(d)| = + \infty$.
Thus, we only need to consider the case that 
	$f^b(B_k(d)) = \sum_{B' \subseteq B_k(d) , \; |B'|=b} 
	\left( \prod_{i \in B'} r^{(d)}_i \right)$.
If we introduce a vector of variables
	$\Vec{r'}=(r'_1, r'_2, \ldots , r'_j)$ where $j=|B_k(d)|$,  
	then the value $f^b(B_k(d))$ is bounded from below 
	by the optimal value of an optimization problem 
\begin{eqnarray} \label{ConcaveMinimization}
	\min \left\{ f^b(\Vec{r}') \left|
		0 \leq r'_i \leq 2^{-d} \; (\forall i \in \{1,2,\ldots , j\}),\;\;
		r'_1+ r'_2+ \cdots + r'_j \geq \lambda_k - 2^{-d}
	\right. \right\}, 
\end{eqnarray}
where 
	$f^b(\Vec{r}') = \sum_{1\leq i_1 < i_2 < \cdots <i_b \leq j} 
		r'_{i_1} r'_{i_2} \cdots r'_{i_b}$ and 
	the last inequality constraint is obtained from~(\ref{LBb2}).
The optimization problem~(\ref{ConcaveMinimization})
	minimizes a continuous function over a bounded closed set 
	(a compact set)
	and has an optimal solution,
	since a given vector of odds $(r^{(d)}_i \mid i \in B_k(d)) $ 
	(arranged in increasing order of indices)
	is feasible for~(\ref{ConcaveMinimization}).
Let  $\Vec{r}^*=(r^*_1, r^*_2, \ldots , r^*_j)$  
	be an optimal solution of~(\ref{ConcaveMinimization}). 
From the symmetry of the objective function and the feasible region, 
	we can assume that $\Vec{r}^*$ satisfies 
	$r^*_1 \geq r^*_2 \geq \cdots \geq r^*_j$.
Now we show that if $r^*_{i+1}>0$,  then $r^*_i=2^{-d}$.
Assume on the contrary that $2^{-d} > r^*_i$ and $r^*_{i+1}>0$ hold.
Then the vector  
	$(r^*_1, r^*_2, \ldots, 
		r^*_i+ \varepsilon, r^*_{i+1}-\varepsilon , r^*_{i+2}, \ldots , r^*_{j})$
	is feasible to~(\ref{ConcaveMinimization}) 
	and
	strictly decreases the objective function value, 
	for some sufficiently small positive $\varepsilon$. 
This contradicts with the optimality of $\Vec{r}^*$.
Thus, the optimal solution $\Vec{r}^*$ is denoted by  
	$( 2^{-d}, 2^{-d}, \ldots, 2^{-d},r'',0,\ldots ,0)$
	where $0 \leq r'' \leq 2^{-d}$.
It is easy to see that $\Vec{r}^*$
	satisfies the last constraint with equality, i.e., 
	$r^*_1+r^*_2+\cdots +r^*_j=\lambda_k - 2^{-d}$. 
Consequently, the number of non-zero elements 
	in $\Vec{r}^*$ is greater than or equal to $2^d\lambda_k-1$.
Thus, we have that
	$\forall b \in \{1,2,\ldots,m\},$
\begin{eqnarray*}
	\lim_{d \rightarrow \infty} f^b(B_k(d)) 
&\geq &
	\lim_{d \rightarrow \infty} f^b(\Vec{r}^*) 
 \geq 
	\lim_{d \rightarrow \infty} 
	\left( 
		\begin{array}{c} 2^d \lambda_k -2 \\ b \end{array}  
	\right) (2^{-d})^b  \\
&\geq &  
	\lim_{d \rightarrow \infty} 
	\frac{(2^d \lambda_k -2 -b)^b}{b!} (2^{-d})^b
	=
	\lim_{d \rightarrow \infty} 
	\frac{(\lambda_k -(2 +b)/2^d )^b}{b!} 
	= \frac{\lambda_k^b}{b!}.
\end{eqnarray*}

\noindent (iv) 
Lastly, we show our lower bound.
As shown in Lemma~\ref{subdivisionLB},
	for any positive integer $d$,  
	the probability of win of any threshold strategy 
	for the $m$-stopping odds problem 
	defined on ${\cal Y}^d$ 
	gives a lower bound of that 
	defined on the given sequence ${\cal X}$. 
Lower bounds derived in (ii) and (iii) directly imply that
	the win probability 
	of {\tt Threshold}$(j_d^{(m)},j_d^{(m-1)},\ldots,j_d^{(1)})$
	gives a lower bound
\[
\begin{array}{l}
\displaystyle
\lim_{d \rightarrow \infty}
	\sum_{k=1}^m 
	\left(
		\left( 
			\prod_{i \in B_k \cup \cdots \cup B_2 \cup B_1 }
			\!\!\!\! q_i 
		\right)
		\sum_{(b_k,\ldots, b_1) \in \Xi_k}\!\!\!\!
		\left( 
			f^{b_k}(B_k(d))  \cdots f^{b_1}(B_1(d))
		\right)
	\right) \\
\displaystyle
\geq \sum_{k=1}^m 
	 \left(
		\left(
			\prod_{k'=1}^k e^{-\lambda_{k'}}
		\right)
		\left(
			\sum_{(b_k,\ldots, b_1) \in \Xi_k}
			\frac{\lambda_{k  }^{b_k}    }{b_k!    } 
			\frac{\lambda_{k-1}^{b_{k-1}}}{b_{k-1}!} 
			\cdots 
			\frac{\lambda_{1  }^{b_1    }}{b_1!} 
		\right) 
	\right)
= \sum_{k=1}^m  e^{- \sum_{k'=1}^k \lambda_{k'}}
\end{array}
\]
where the last equality is obtained from~(\ref{EQSystem}).
\hfill \qed
\smallskip

\begin{table}
\caption{Cumulative sum of entries 
	in $(\lambda_1, \lambda_2, \ldots, \lambda_m)$
	and lower bounds.} \label{tablesumlambda}
\[
\begin{array}{c||rcl|c}
m & \displaystyle \sum_{k=1}^m \lambda_k &&& 
 \displaystyle \sum_{k=1}^m  e^{- \sum_{k'=1}^k \lambda_{k'}} \\
\hline 
1 & 1 		&=& 1 					& 0.3678794411\cdots \\ 
2 & 3/2 	&=& 1.5 				& 0.5910096013\cdots \\
3 & 47/24	&=& 1.958333 \cdots 	& 0.7321029820\cdots \\
4 & 2761/1152&=& 2.396701 \cdots 	& 0.8231206726\cdots \\
5 &	4162637/1474560 &=& 2.822969 \cdots	& 0.8825499145\cdots \\
\hline
6 & 380537052235603/117413668454400 &=& 3.240994 \cdots
									& 0.9216748810\cdots \\[+1.5ex]
7 & \displaystyle 
\frac{705040594914523588948186792543}%
{193003573558876719588311040000}
	&=& 3.652992 \cdots				& 0.9475883491\cdots \\[+2.5ex]
8 & \displaystyle \frac{
\begin{array}{r}
3025002101 \\[-0.5ex]
77484374840641189918370275991590974715547528765249 \\[-0.5ex] 
\end{array}}%
{
\begin{array}{r}
745007588 \\[-0.5ex]
12993473612938854416966977838930799571763200000000 \\[-0.5ex]
\end{array}
}
\!\!&=& 4.060364 \cdots				&  0.9648310882\cdots \\[4.5ex]
9  & \displaystyle \frac{ 
\begin{array}{r}
4955429267826902943  \\[-0.5ex] 
22991702889058732983678465397265103848504031927299  \\[-0.5ex] 
12522937262239403638817695466470734534217406992001  \\[-0.5ex] 
\end{array}}%
{
\begin{array}{r} 
1110072612742364945  \\[-0.5ex]  
47845493213273623476317581768829551455545915219181  \\[-0.5ex] 
23315624957195621435513013513748480000000000000000  \\[-0.5ex]
\end{array}
}
\!\!&=&4.464059 \cdots				&  0.9763466188\cdots  \\[+7.5ex]
10 & \displaystyle \frac{
\begin{array}{r}
11989289035379212035246168789525873032  \\[-0.5ex]
80849078486814692748999764352069320540924554366342  \\[-0.5ex]
20167531781129657310860185112917637526070528528590  \\[-0.5ex]
30333616681207477435841890935057636581590554638168  \\[-0.5ex]
66245450807944253110095088765765115912740477984001  \\[-0.5ex]
\end{array}}%
{
\begin{array}{r}
2464522411121321065440235649497935699  \\[-0.5ex]
29104324006809890312129945082092425080632748539150  \\[-0.5ex]
55027199035483748007701106129537805807977992169375  \\[-0.5ex]
27132513498662936774805285136477456358447232852772  \\[-0.5ex]
52323755961404620800000000000000000000000000000000  \\[-0.5ex]
\end{array}
}
\!\!&= &4.864751 \cdots				& 0.9840603638\cdots \\
\hline 
\end{array}
\]
\end{table}



\section{Secretary Problem.} \label{section:CSP}

In this section, we show that an optimal strategy 
	for the secretary problem attains our lower bounds 
	appearing in Theorem~\ref{mainLB}.
We discuss a sequence of $0/1$ random variables 
$X_2, X_3, \ldots , X_N$ satisfying
$ \Prob [X_i=1]=1/i$, for any $i \in \{2,3,\ldots,N\}$.
In this section, $q_i$ denotes the probability of failure $1-1/i$ 
	and $r_i$ denotes the odds $1/(i-1)$ of $X_i$, 
	for all \mbox{$i \in \{1,2,\ldots ,N\}$}.
Gilbert and Mosteller~\cite{GM1966} showed that
	an optimal strategy 
	for the secretary problem with $m$-stoppings 
	is attained by a threshold strategy.

First, we show some properties
	related to the optimal threshold strategy.

\begin{lemma} \label{CSP}
Consider an $m$-stopping secretary problem
	defined on Bernoulli sequence $X_2, X_3, \ldots , X_N$  
	satisfying $\Prob [X_i=1]=1/i$ 
	$(\forall i \in \{2,3,\ldots,N\})$.
We denote an optimal threshold strategy by
	${\tt Threshold} (i_N^{(m)},$ $i_N^{(m-1)}, \ldots, i_N^{(1)})$,
	and define a block partition 
	$\{ B_{m+1}(N),B_m(N),\ldots,B_1(N)\}$
	of index set $\{2,3,\ldots,N\}$ by
\[
B_k(N)= \left\{ \begin{array}{ll}
	\{i \in \{2,3,\ldots,N \} \mid i_N^{(1)} \leq i \leq N \} 
			& (k=1), \\
	\{i \in \{2,3,\ldots,N \} \mid i_N^{(k)} \leq i < i_N^{(k-1)} \} 
			&(1<k \leq m), \\
	\{i \in \{2,3,\ldots,N \} \mid 2 \leq i<i_N^{(m)} \}
			&(k=m+1).
			\end{array} \right.
\]
	
Then, for each $k \in \{1,2,\ldots,m\}$,
	the following properties hold;
	\mbox{
	{\rm (i)}~$\lim_{N \rightarrow \infty} i_N^{(k)}= + \infty$, 
	}
	{\rm (ii)}~$\lim_{N \rightarrow \infty} 
			\sum_{i \in B_k(N)} r_i=\lambda_k$,
	{\rm (iii)}~$\lim_{N \rightarrow \infty} 
			\prod_{i \in B_k(N)}q_i = e^{-\lambda_k}$, 
	and
	{\rm (iv)}~$\lim_{N \rightarrow \infty} f^b(B_{k}(N) )
			= \frac{\lambda_{k}^b}{b!} $ $\;\;\; 			
		(\forall b \in \{0,1,\ldots,m\})$  
	where $(\lambda_1,\lambda_2,\ldots,\lambda_m)$ 
	is a unique solution of~\mbox{\rm (\ref{EQSystem})}.
\end{lemma}

 \mbox{Proof. }
It is well-known that for any $ k \in \{1,2,\ldots , m\}$, 
	the threshold strategy 
	${\tt Threshold} (i_N^{(k)},$ $ i_N^{(k-1)}, \ldots, i_N^{(1)})$
	is also optimal 
	to the $k$-stopping secretary problem~\cite{GM1966}. 

In the following, we show desired properties by induction on $k$.
When $k=1$, the problem becomes the classical secretary problem, 
	and properties (i)$,\ldots,$(iv) are well-known
	(see~\cite{GM1966} for example). 

Now, we begin a discussion of the $k$-th induction step 
	(where $k\leq m$) under an assumption that 
	for any $k' \in \{1,2,\ldots, k-1 \}$,
	properties (i)$,\ldots,$(iv) hold.	
Let $\Vec{e}$ be a unit $k$-vector $(1,0,0,\ldots,0)$.
Lemma~\ref{unitvector} says that 
	$\Vec{e} \in \Xi_k$ and 
	every vector $(b_k,b_{k-1},\ldots,b_1) \in \Xi_k \setminus \{\Vec{e}\}$
	satisfies $b_k=0$.
Thus, the induction hypothesis (iv) 
	and the definition of equality system~(\ref{EQSystem}) imply that
\begin{eqnarray}  \hspace{5ex}
\nonumber
\lefteqn{
	\lim_{N \rightarrow \infty}
	\left( \sum_{(b_k,\ldots, b_1) \in \Xi_k \setminus \{{\small \Vec{e}} \}}
		\left(
			f^{b_{k-1}}(B_{k-1}(N)) f^{b_{k-2}}(B_{k-2}(N)) 
			\cdots f^{b_1}(B_1(N))
		\right)
	\right) } \\ 
 &=&
	\sum_{(b_k,\ldots ,b_1) \in \Xi_k \setminus \{{\small \Vec{e}}\}}
	\left( 
		\frac{\lambda_{k-1}^{b_{k-1}}}{b_{k-1}!    } 
		\frac{\lambda_{k-2}^{b_{k-2}}}{b_{k-2}!} 
		\cdots 
		\frac{\lambda_{1  }^{b_1    }}{b_1!} 
	\right)  
=
	\sum_{(b_k,\ldots ,b_1) \in \Xi_k \setminus \{{\small \Vec{e}}\}}
	\left( 
 		\frac{\lambda_{k  }^{0      }}{0!      } 
		\frac{\lambda_{k-1}^{b_{k-1}}}{b_{k-1}!    } 
		\frac{\lambda_{k-2}^{b_{k-2}}}{b_{k-2}!} 
		\cdots 
		\frac{\lambda_{1  }^{b_1    }}{b_1!} 
	\right) 
  =	1-\lambda_k.  \label{tightlambdaCSP}
\end{eqnarray}

Now we introduce a threshold strategy
	{\tt Threshold}$(i, i_N^{(k-1)}, i_N^{(k-2)}, \ldots, i_N^{(1)})$
	$\;\;\; (\forall i \in \{2,3,\ldots, i_N^{(k-1)}\})$
	for the $k$-stopping secretary problem 
	and employ a one-stage look-ahead approach~\cite{ANO2001}.
Let $\Pwin_N (k, i)$ be a probability of win 
	of the threshold strategy 
	{\tt Threshold}$(i, i_N^{(k-1)}, i_N^{(k-2)}, \ldots, i_N^{(1)})$.

A difference of a pair  $\Pwin_N (k, i-1)$ and $\Pwin_N (k, i)$
satisfies
\begin{eqnarray}
\lefteqn{
       \Pwin_N (k, i-1) - \Pwin_N (k, i)
}  \nonumber \\ 
&=&
       \left( \prod_{i'=i-1}^{N} q_{i'} \right)
       \left( \sum_{i' =i-1}^{i_N^{(k-1)}-1} r_{i'} +
               \sum_{(b_k,\ldots, b_1) \in \Xi_k \setminus \{{\small \Vec{e}}\}}
               \left(
                       f^{b_{k-1}}(B_{k-1}(N)) \cdots f^{b_1}(B_1(N))
               \right)
       \right) \nonumber \\
&& -
       \left( \prod_{i'=i}^{N} q_{i'} \right)
       \left( \sum_{i'=i}^{i_N^{(k-1)}-1}r_{i'} +
               \sum_{(b_k,\ldots, b_1) \in \Xi_k \setminus \{{\small \Vec{e}}\}}
               \left(
                       f^{b_{k-1}}(B_{k-1}(N)) \cdots f^{b_1}(B_1(N))
               \right)
       \right) \nonumber \\
&=& \left( \prod_{i'=i}^{N} q_{i'} \right)
       \left(
       \begin{array}{l}
               \displaystyle
               ( q_{i-1}-1 ) \sum_{i'=i}^{i_N^{(k-1)}-1}r_{i'}
                       + q_{i-1}r_{i-1} \\[1ex]
               \displaystyle
               + (q_{i-1} -1)
               \sum_{(b_k,\ldots, b_1) \in \Xi_k \setminus \{{\small \Vec{e}}\}}
               \left(
                       f^{b_{k-1}}(B_{k-1}(N)) \cdots f^{b_1}(B_1(N))
               \right)
       \end{array}
       \right) \nonumber \\
&=& \left( \prod_{i'=i}^{N} q_{i'} \right)
       (1- q_{i-1})
       \left(
       \begin{array}{l}
               \displaystyle
               - \sum_{i'=i}^{i_N^{(k-1)}-1}r_{i'}  + 1 \\[1ex]
               \displaystyle
               - \sum_{(b_k,\ldots, b_1) \in \Xi_k \setminus \{{\small \Vec{e}}\}}
               \left(
                       f^{b_{k-1}}(B_{k-1}(N))  \cdots f^{b_1}(B_1(N))
               \right)
       \end{array}
       \right). \label{differenceCSP} 
\end{eqnarray}
%

It is easy to show that 
\[
	\lim_{N \rightarrow \infty}
	\left( 
		\Pwin_N (k,3) -  \Pwin_N (k,2) 
	\right)
	> 0 > 
	\lim_{N \rightarrow \infty}
	\left( 
		\Pwin_N (k,i_N^{(k-1)})- \Pwin_N (k,i_N^{(k-1)}-1) 
	\right),
\]
which implies that $\exists N', \forall N > N'$, 
	$3 \leq i^{(k)}_N \leq i_N^{(k-1)}-1$.
From the optimality of the threshold value $i^{(k)}_N$,  
	inequalities 
		$\Pwin_N (k,i^{(k)}_N-1) 
	\leq \Pwin_N (k,i^{(k)}_N)
	\geq \Pwin_N (k,i^{(k)}_N+1)  $
	hold.

\smallskip

\noindent
(i)	Since $0 \leq \Pwin_N (k,i^{(k)}_N) - \Pwin_N (k,i^{(k)}_N+1) $, 
	equality~(\ref{differenceCSP}) implies that 
\[
	\sum_{i'= i^{(k)}_N+1}^{i_N^{(k-1)}-1} r_{i'}  \leq   
		1 -   \sum_{(b_k,\ldots, b_1) \in \Xi_k \setminus \{{\small \Vec{e}}\}}
		\left(
			f^{b_{k-1}}(B_{k-1}(N))  \cdots f^{b_1}(B_1(N))
		\right).
\]
As a consequence of the above inequality and~(\ref{tightlambdaCSP}), 
	we have that 
\begin{eqnarray}
\label{upperBlock}
\lambda_k 
&= & \lim_{N \rightarrow \infty}
	\left(
		1- \hspace{-3ex} \sum_{(b_k,\ldots, b_1) \in \Xi_k \setminus \{{\small \Vec{e}}\}}
		\hspace{-3ex}
		\left(
			f^{b_{k-1}}(B_{k-1}(N))  \cdots f^{b_1}(B_1(N))
		\right)
	\right) 
\geq 	\lim_{N \rightarrow \infty}
		\left(
			\sum_{i'=i^{(k)}_N+1}^{i_N^{(k-1)}-1} r_{i'}
		\right)	\\   
\nonumber
 &=&		\lim_{N \rightarrow \infty}
 		\left(
 			\sum_{i'=i^{(k)}_N}^{i_N^{(k-1)}-2} \frac{1}{i'}
 		\right)  
 \geq  	\lim_{N \rightarrow \infty}
		\left(
			\ln \frac{i_N^{(k-1)}-1}{i^{(k)}_N}
		\right).
\end{eqnarray}
The above inequality and the assumption 
	that $\lim_{N \rightarrow \infty} i^{(k-1)}_N = +\infty$
	imply the property that
	$\lim_{N \rightarrow \infty} i^{(k)}_N = +\infty$. 
	
\smallskip
\noindent
(ii)
Form inequality~(\ref{upperBlock}) and 
	the property $\lim_{N \rightarrow \infty} i^{(k)}_N = +\infty$,
	we obtain that 
\begin{eqnarray*}
			\lim_{N \rightarrow \infty} \sum_{i' \in B_k(N)} r_{i'} 
	=		\lim_{N \rightarrow \infty} \sum_{i'=i_N^{(k)}}^{i_N^{(k-1)}-1} r_{i'} 
	= 		\lim_{N \rightarrow \infty}
				\left( 
					r_{i^{(k)}_N} + \sum_{i'=i_N^{(k)}+1}^{i_N^{(k-1)}-1} r_{i'} 
			\right)  
\leq 	\lim_{N \rightarrow +\infty} 
	\left( \frac{1}{i^{(k)}_N-1} \right) + \lambda_k  = \lambda_k.
\end{eqnarray*}

\noindent
The inequality 
	$0 \geq \Pwin_N (k,i^{(k)}_N-1)- \Pwin_N (k,i^{(k)}_N) $
	and equality~(\ref{differenceCSP}) directly imply that
\[
	\sum_{i'=i^{(k)}_N}^{i_N^{(k-1)}-1} r_{i'}
  \geq  
		1 -   \sum_{(b_k,\ldots, b_1) \in \Xi_k \setminus \{{\small \Vec{e}}\}}
		\left(
			f^{b_{k-1}}(B_{k-1}(N))  \cdots f^{b_1}(B_1(N))
		\right).
\]
Combining the above inequality and~(\ref{tightlambdaCSP}), 
	we obtain that
\begin{eqnarray*}
	\lim_{N \rightarrow \infty} \sum_{i' \in B_k(N)} r_{i'}
=	\lim_{N \rightarrow \infty}
			\sum_{i'=i^{(k)}_N}^{i_N^{(k-1)}-1} r_{i'}
 \geq  
	\lim_{N \rightarrow \infty}
	\left(
		1- \hspace{-3ex} \sum_{(b_k,\ldots, b_1) \in \Xi_k \setminus \{{\small \Vec{e}}\}}
		\left(
			f^{b_{k-1}}(B_{k-1}(N))  \cdots f^{b_1}(B_1(N))
		\right)
	\right) 
= \lambda_k.
\end{eqnarray*}

As a result, we have shown that
	$
		\lim_{N \rightarrow \infty} \sum_{i \in B_k(N)} r_{i}
		=\lambda_k.
	$
	
\smallskip
\noindent 
(iii) 
From the definition of $q_i$, it is clear that
\begin{eqnarray}
\nonumber
	  \ln \prod_{i \in B_k(N)} q_i 
	&=& \sum_{i \in B_k(N)} \ln (1-1/i)
	= \sum_{i \in B_k(N)} \left( \ln (i-1) - \ln i \right) \\
&&	= \ln (i_N^{(k)}-1) - \ln (i_N^{(k-1)}-1) 
	= \ln \frac{i_N^{(k)}-1}{i_N^{(k-1)}-1} 
\label{ratio-threshold} 
\end{eqnarray}
and
\begin{eqnarray*}
	\sum_{i=i_N^{(k)}-1}^{i_N^{(k-1)}-2} \frac{1}{i} \geq 
& \displaystyle
	 \ln (i_N^{(k-1)}-1) - \ln (i_N^{(k)}-1)  
&\geq 	\sum_{i=i_N^{(k)}}^{i_N^{(k-1)}-2} \frac{1}{i},
\end{eqnarray*}
where $i_N^{(0)}$ denotes $N+1$.
From the above, we obtain that
\begin{eqnarray*}
- \lim_{N \rightarrow \infty} \ln \prod_{i \in B_k(N)} q_i 
&\leq & 	\lim_{N \rightarrow \infty} \sum_{i=i_N^{(k)}-1}^{i_N^{(k-1)}-2} \frac{1}{i}
	=	\lim_{N \rightarrow \infty} \sum_{i=i_N^{(k)}}^{i_N^{(k-1)}-1} r_i 
	=	\lim_{N \rightarrow \infty} \sum_{i=B_k(N)} r_i
	=	\lambda_k,	 \;\;\; \mbox{and} \\
- \lim_{N \rightarrow \infty} \ln \prod_{i \in B_k(N)} q_i
&\geq &	\lim_{N \rightarrow \infty} \sum_{i=i_N^{(k)}}^{i_N^{(k-1)}-2} \frac{1}{i}
	=	\lim_{N \rightarrow \infty}
	\left(
		\left( \frac{-1}{i^{(k)}_N-1} \right)
		+ \sum_{i=i_N^{(k)}}^{i_N^{(k-1)}-1} r_i
	\right)
	=	\lim_{N \rightarrow \infty} \sum_{i=B_k(N)} r_i
	=	\lambda_k.
\end{eqnarray*}
Accordingly, we have that 
	$ \lim_{N \rightarrow \infty} \prod_{i \in B_k(N)} q_i =e^{-\lambda_k}$.

\smallskip
\noindent 
(iv) 
We omit the case of $f^0(B_k(N))$ in the following,
	since the equality $f^0(B_k(N))=1=\frac{\lambda_k^0}{0!}$ holds permanently.
We discuss cases $b \in  \{1,2,\ldots, m\}$.
The size of block $B_k(N)$ satisfies the following; 
\[
	\lim_{N \rightarrow \infty} |B_k (N)|r_{i_N^{(k)}} 
	\geq \lim_{N \rightarrow \infty} \sum_{i=i_N^{(k)}}^{i_N^{(k-1)}-1} r_{i}
	= \lim_{N \rightarrow \infty} \sum_{i \in B_k(N)} r_{i}
	=  \lambda_k >0.
\]
The positivity of $\lambda_k$ 
	and the equality 
	$\lim_{N \rightarrow \infty}r_{i_N^{(k)}}=\lim_{N \rightarrow \infty}\frac{1}{i_N^{(k)}-1}=0$
	imply that 
	$ \lim_{N \rightarrow \infty} |B_k(N)| = +\infty$.
Thus, we have that $\exists N'$, $\forall N > N'$, 
	the size of $B_k(N)$ exceeds $m$  and thus
\[
	\forall b \in \{1,2,\ldots,m\}, \;\;\; 
	f^b(B_{k}(N) )= 
		\sum_{B' \subseteq B_k(N), \; |B'|=b} 
		\left( \prod_{i \in B'} r_i \right).
\]
We show that 
	$\lim_{N \rightarrow \infty} f^b(B_{k}(N) )= \frac{\lambda_{k}^b}{b!}$
	by induction on $b$.
When  $b=1$, property~(ii) implies that
\[
	\lim_{N \rightarrow \infty} f^1 (B_{k}(N) )
		= \lim_{N \rightarrow \infty} \sum_{i \in B_k(N)} r_i
		= \frac{\lambda_{k}^{1}}{1!}. 
\]
For any positive vector $\Vec{r}' \in \Real^{n'}$, 
	the inequality $f^b(\Vec{r}') \leq (\sum_{i=1}^{n'} r'_i)/(b!)$ holds, 
	and thus we have 
\begin{eqnarray} \label{ToReferee-1}
	\lim_{N \rightarrow \infty} f^b(B_k(N))
	\leq \lim_{N \rightarrow \infty} 
		\frac{\left( \sum_{i \in B_k(N)} r_i \right)^b}{b!}
	= \frac{\lambda_{k}^b}{b!}. 
\end{eqnarray}
The induction hypothesis  
	$\lim_{N \rightarrow \infty} f^{b-1}(B_{k}(N) )
	= \frac{\lambda_{k}^{b-1}}{(b-1)!}$
	implies that
\begin{eqnarray}
\lefteqn{
	  \lim_{N \rightarrow \infty} f^b(B_k(N))
	=\lim_{N \rightarrow \infty} 
		\sum_{B' \subseteq B_k(N), \; |B'|=b} 
		\left( \prod_{i \in B'} r_i \right)  
}\nonumber \\
	&=& \lim_{N \rightarrow \infty} 
		\left( \frac{1}{b} \right) 
		\sum_{B'' \subseteq B_k(N), \; |B''|=b-1} 
		\left( 
			\left( \prod_{i \in B''} r_i \right)
			\left( \sum_{j \in B_k(B) \setminus B''} r_j \right)
		\right) \nonumber \\
	&\geq & \lim_{N \rightarrow \infty} 
		\left( \frac{1}{b} \right) 
		\left( \sum_{i \in B_k(N)} r_i - (b-1) r_{i_N^{(k)}}\right)
		\sum_{B'' \subseteq B_k(N), \; |B''|=b-1} 
		\left( \prod_{i \in B''} r_i \right) \nonumber \\
	&= & \lim_{N \rightarrow \infty} 
		\frac{
		\left( \sum_{i \in B_k(N)} r_i - (b-1) r_{i_N^{(k)}}\right)}{b} 
		f^{b-1}(B_k (N)) 
	= \frac{\lambda_k}{b} \frac{\lambda_k^{b-1}}{(b-1)!} 
		= \frac{\lambda_k^b}{b!}.
	\label{ToReferee-2}
\end{eqnarray}
Thus, we have shown 
$
	 \lim_{N \rightarrow \infty } f^b(B_{k}(N) )
	= \frac{\lambda_{k}^b}{b!}.
$
\hfill \qed
\smallskip

The following theorem gives the win probability of the secretary problem.

\begin{theorem} \label{CSP1}
Let $(\lambda_1,\lambda_2,\ldots,\lambda_m)$ 
	be a unique solution of~\mbox{\rm (\ref{EQSystem})}.
Given a sequence 
	of $0/1$ random variables  $X_2, X_3, \ldots , X_N$  
	satisfying $\Prob [X_i=1]=1/i$ 
	$(\forall i \in \{2,3,\ldots,N\})$,
	the probability of win $\Pwin_N (m)$ of an optimal strategy 
	for the $m$-stopping secretary problem
	defined on $X_2, X_3, \ldots , X_N$  
	satisfies 
\[
	\lim_{N \rightarrow \infty} \Pwin_N (m) =
	\sum_{k=1}^m e^{-\sum_{k'=1}^k \lambda_{k'}}.
\]
\end{theorem}

 \mbox{Proof. }
In the previous lemma, 
	we have shown properties~(iii) and~(iv) 
	for each $k \in \{1,2,\ldots, m\}$.
Thus, the probability of win $\Pwin_N (m)$ satisfies 
\begin{eqnarray*}
	\lim_{N \rightarrow \infty} \Pwin_N (m)
&= & 
	\lim_{N \rightarrow \infty} 
	\sum_{k=1}^m 
	\left( \left( \prod_{i \in B_k(N) \cup \cdots \cup B_1(N)} \!\!\! q_i \right)
		\left( \sum_{(b_k,\ldots, b_1) \in \Xi_k } \!\!\!\!\!\!
			\left(
				f^{b_{k}}(B_{k}(N))  \cdots f^{b_1}(B_1(N))
			\right)
		\right) 
	\right) \\
&= &
	\sum_{k=1}^m 
	\left( \left( \prod_{k'=1}^k e^{-\lambda_{k'}} \right)
		\left( \sum_{(b_k,\ldots, b_1) \in \Xi_k }
			\frac{\lambda_{k  }^{b_{k  }}}{b_{k  }!} 
			\frac{\lambda_{k-1}^{b_{k-1}}}{b_{k-1}!} 
			\cdots 
			\frac{\lambda_{1  }^{b_1    }}{b_1!} 
		\right)
	\right) 
= \sum_{k=1}^m e^{-\sum_{k'=1}^k \lambda_{k'}}, 
\end{eqnarray*} 
where the last equality is obtained from~(\ref{EQSystem}).
\hfill \qed
\smallskip


In the rest of this section, 
	we prove a conjecture on a relation between
	threshold values and  win probability 
	of the secretary problem
	indicated by Gilbert and Mosteller~\cite{GM1966}.

\begin{theorem}
Given a sequence 
	of $0/1$ random variables  $X_2, X_3, \ldots , X_N$  
	satisfying $\Prob [X_i=1]=1/i$ 
	$(\forall i \in \{2,3,\ldots,N\})$,
	an optimal (threshold) strategy 
	{\tt Threshold}$(i_N^{(m)},i_N^{(m-1)},\ldots,i_N^{(1)})$
	for the $m$-stopping secretary problem satisfies
\[
	\lim_{N \rightarrow \infty}
	\left( 
		\frac{i_N^{(m)}}{N}+\frac{i_N^{(m-1)}}{N}+\cdots + \frac{i_N^{(1)}}{N} 
	\right)
	= \lim_{N \rightarrow \infty} \Pwin_N (m)
\]
	where $\Pwin_N (m)$ denotes a corresponding probability of win. 
\end{theorem}

 \mbox{Proof. }
In the following, 
	we put $i_N^{(0)}=N+1$ and
	we denotes a unique solution of~\mbox{\rm (\ref{EQSystem})}
	by $(\lambda_1,\lambda_2,\ldots,\lambda_m)$.
Equality~(\ref{ratio-threshold}) 
	imply that for any $k \in \{1,2,\ldots, m\}$, 
\begin{eqnarray*}
\lefteqn{
		\lim_{N \rightarrow \infty} \ln \frac{i_N^{(k)}}{i_N^{(k-1)}}
	= 	\lim_{N \rightarrow \infty} 
		\ln \left(
			\frac{i_N^{(k)  }-1}{i_N^{(k-1)}-1} \;
			\frac{i_N^{(k-1)}-1}{i_N^{(k-1)}  } \;
			\frac{i_N^{(k)  }  }{i_N^{(k)  }-1}
		\right) 
}\\
	&=& \lim_{N \rightarrow \infty} 
			\left( 
				  \ln \left( \prod_{i \in B_k(N)} q_i \right)
				+ \ln \left( 1- \frac{1}{i_N^{(k-1)}} \right)
				+ \ln \left( 1+ \frac{1}{i_N^{(k)}-1} \right) 
			\right)
	= -\lambda_k.
\end{eqnarray*}
Thus, we have 
\begin{eqnarray*}
	\lim_{N \rightarrow \infty} \ln \frac{i_N^{(k)}}{N}
	&=&
	\lim_{N \rightarrow \infty} \ln
		\left(
			\left( \frac{i_N^{(k)}}{i_N^{(k-1)}} \right)
			\left( \frac{i_N^{(k-1)}}{i_N^{(k-2)}} \right)
			\cdots
			\left( \frac{i_N^{(1)}}{i_N^{(0)}} \right) 
			\left( \frac{N+1      }{N} \right) 
		\right)	\\	
	&=& 
	\lim_{N \rightarrow \infty} 
		\left(
			\ln \left( \frac{i_N^{(k)}}{i_N^{(k-1)}} \right)
		+	\ln \left( \frac{i_N^{(k-1)}}{i_N^{(k-2)}} \right)
		+	\cdots
		+	\ln \left( \frac{i_N^{(1)}}{i_N^{(0)}} \right) 
		+	\ln \left( \frac{N+1      }{N} \right) 
		\right)	\\
	&=& -(\lambda_{k}+\lambda_{k-1}+\cdots + \lambda_1)
	= - \sum_{k'=1}^k \lambda_{k'}. 
\end{eqnarray*}
The above equality and Theorem~\ref{CSP1} imply that
\[
	\lim_{N \rightarrow \infty}
	\left( 
		\frac{i_N^{(m)}}{N}+\frac{i_N^{(m-1)}}{N}+\cdots + \frac{i_N^{(1)}}{N} 
	\right)
	= \lim_{N \rightarrow \infty} \sum_{k=1}^m \frac{i_N^{(k)}}{N}
	= \sum_{k=1}^m e^{- \sum_{k'=1}^k \lambda_{k'}}
	= \lim_{N \rightarrow \infty} \Pwin_N (m).
\]
\hfill \qed
\smallskip

\section{Discussion.}

We dealt with the odds problem and the secretary problem
	with multiple stopping chances.
We derived a tight lower bound of the probability of win 
	for odds problem and showed that 
	the lower bound is attained by the secretary problem.
We also proved a conjecture on the secretary problem
	which connects the optimal threshold strategy and the probability of win.

\section*{APPENDIX}{Proof of Theorem~\ref{PositiveSol}}

Before showing Theorem~\ref{PositiveSol}, 
	we need some definitions and a lemma.
For any pair of positive integers $(k,k')$ 
	satisfying $m\geq k > k' \geq 1$,
	we introduce
\[
	\Xi^0_k (k')  =
	\left\{ 
		(b_k,b_{k-1},\ldots ,b_1) \in \Xi_k   
		\left| \;   
			0=b_{k}=b_{k-1}=\cdots =b_{k'+1} 
		\right. 
	\right\},
\]
where we define $\Xi^0_k(k)=\Xi_k$.
We also denote 
\[
	\alpha (k,k')
	=	\sum_{(b_k,\ldots ,b_1) \in \Xi^0_k(k')}
		\left( 
			\frac{\lambda_{k  }^{b_k}    }{b_k!    } 
			\frac{\lambda_{k-1}^{b_{k-1}}}{b_{k-1}!} 
			\cdots 
			\frac{\lambda_{1  }^{b_1    }}{b_1!} 
		\right)
	=	\sum_{(0,\ldots, 0, b_{k'},\ldots ,b_1) \in \Xi^0_k(k')}
		\left( 
			\frac{\lambda_{k'  }^{b_{k'}}  }{b_{k'  }!} 
			\frac{\lambda_{k'-1}^{b_{k'-1}}}{b_{k'-1}!} 
			\cdots 
			\frac{\lambda_{1  }^{b_1    }}{b_1!} 
		\right),
\]
	where $(\lambda_m,\lambda_{m-1},\ldots, \lambda_1)$ is a solution
	of equality system~(\ref{EQSystem}).

\begin{lemma}\label{alpha}
For any integer $k' \in \{1,2,\ldots , m\}$, 
	$\alpha (\cdot, \cdot)$ satisfies
	the following inequalities
\[
  1=\alpha (k',k')>\alpha (k'+1,k') > \cdots > \alpha (m,k') >0.
\]
\end{lemma}

 \mbox{Proof. }
When we consider $\alpha (k',k')$, 
	equality $\Xi^0_{k'}(k')=\Xi_{k'}$ holds and thus
	(\ref{EQSystem}) directly implies that
\[
	\alpha (k',k')
=	\sum_{(b_{k'},\ldots ,b_1) \in \Xi_{k'}}
		\left( 
			\frac{\lambda_{k'  }^{b_{k'}}  }{b_{k'  }!} 
			\frac{\lambda_{k'-1}^{b_{k'-1}}}{b_{k'-1}!} 
			\cdots 
			\frac{\lambda_{1  }^{b_1    }}{b_1!} 
		\right)=1.
\]
Next, we show the above inequalities by induction on $k'$.

When $k'=1$, it is clear that $\Xi^0_k(1)=\{(0,0, \cdots, 0, k)\}$
	and thus $\alpha (k,1)=\frac{\lambda_1^k}{k!}=\frac{1}{k!} $ 
	for each $k \in \{1,2,\ldots ,m\}$, 
	since $\lambda_1=1$.
Thus, we have the inequality
\[
	1=\alpha (1,1) > \alpha(2,1) > \cdots > \alpha (m,1) >0.
\] 

Next, we consider a general case.
Every vector $(0,\ldots ,0,b_{k'},\ldots ,b_1) \in \Xi^0_k(k')$
	satisfies that $b_{k'}\in \{0,1,\ldots ,1+k-k'\}$. 
Especially in the case that $b_{k'}=1+k-k'$, 
	the definition of $\Xi_k$ implies that 
	$(0,\ldots ,0, b_{k'},\ldots ,b_1)
	=(0,\ldots ,0, 1+k-k', 0, \ldots ,0)$.
We consider the remaining case that 
	$b_{k'} \in \{0,1,\ldots ,k-k'\}$.
Then, there exists an index $k^* \in \{k'-1,k'-2 \ldots , 2,1\}$
	satisfying 
\[
\begin{array}{lcl}
	1+k-k'  	&>&b_{k'}, \\
	1+k-(k'-1)  &>&b_{k'}+ b_{k'-1}, \\
			& \vdots & \\
	1+k-(k^*+1) &>&b_{k'}+ b_{k'-1}+\cdots + b_{k^*+1},\\			
	1+k-k^* 	&=&b_{k'}+ b_{k'-1}+\cdots + b_{k^*+1}+ b_{k^*}, \\
			0	&=& b_{k^*-1}=\cdots =b_1.
\end{array}
\]
The above inequalities and equalities imply that
\[
\begin{array}{lcl}
	1+(k-b_{k'})- k'  	&>&0, \\
	1+(k-b_{k'})-(k'-1)  &>&b_{k'-1}, \\
			& \vdots & \\
	1+(k-b_{k'})-(k^*+1) &>&b_{k'-1}+\cdots + b_{k^*+1},\\			
	1+(k-b_{k'})-k^* 	&=&b_{k'-1}+\cdots + b_{k^*+1}+ b_{k^*}, \\
			0	&=& b_{k^*-1}=\cdots =b_1,
\end{array}
\]
	and thus we have 
	$(0,\ldots ,0,0, b_{k'-1},\ldots ,b_1) 
		\in \Xi^0_{k-b_{k'}}(k'-1)$.
The inverse implication is clear, i.e.,
	$0 \leq \forall b_{k'} \leq k-k'$, 
	if
	$(0,\ldots ,0,0, b_{k'-1},\ldots ,b_1) 
		\in \Xi^0_{k-b_{k'}}(k'-1)$,
	then 
	$(0,\ldots ,0,b_{k'},\ldots ,b_1) \in \Xi^0_k(k')$.
	
From the above,  
	the definition of $\alpha (k, k')$ implies that
\begin{eqnarray}
 \alpha (k, k') &=&
  	\sum_{(0,\ldots, 0, b_{k'},\ldots ,b_1) \in \Xi^0_k(k')}
	\left( 
		\frac{\lambda_{k'  }^{b_{k'}}  }{b_{k'  }!} 
		\frac{\lambda_{k'-1}^{b_{k'-1}}}{b_{k'-1}!} 
		\cdots 
		\frac{\lambda_{1  }^{b_1    }}{b_1!} 
	\right)  \nonumber \\
&=&		
	\sum_{b=0}^{1+k-k'}
 	\sum_{(0,\ldots, 0, b, b_{k'-1},\ldots ,b_1) \in \Xi^0_k(k')}
	\left( 
		\frac{\lambda_{k'  }^{b       }}{b!       } 
		\frac{\lambda_{k'-1}^{b_{k'-1}}}{b_{k'-1}!} 
		\cdots 
		\frac{\lambda_{1  }^{b_1    }}{b_1!} 
	\right)  \nonumber \\
&=& \frac{\lambda_{k'  }^{1+k-k'}}{(1+k-k')!}
		\left( 
			\frac{\lambda_{k'-1}^{0}}{0!} 
			\cdots 
			\frac{\lambda_{1   }^{0}}{0!} 
		\right) 
  + \sum_{b=0}^{k-k'}
	\left(
		\frac{\lambda_{k'  }^{b       }}{b!       }
		\sum_{(0,\ldots, 0, 0, b_{k'-1},\ldots ,b_1)
				\in \Xi^0_{k-b}(k'-1)}
		\left( 
			\frac{\lambda_{k'-1}^{b_{k'-1}}}{b_{k'-1}!} 
			\cdots 
			\frac{\lambda_{1  }^{b_1    }}{b_1!} 
		\right) 
	\right)
 \nonumber  \\
 &=& \frac{\lambda_{k'  }^{1+k-k'}}{(1+k-k')!}
  +  \sum_{b=0}^{k-k'}
	\left(
		\frac{\lambda_{k'}^{b}}{b!} \; 
		\alpha (k-b,k'-1) 
	\right).
 \label{alpha_K-induction} 
\end{eqnarray}
When $k=k'$, the above equality implies that
\begin{equation}\label{lambda_K-induction}
1= \alpha (k',k')
	=\frac{\lambda_{k'}}{1!}
	+\frac{\lambda_{k'}^0}{0!}\alpha (k',k'-1)
	=\lambda_{k'} + \alpha (k',k'-1)
\end{equation}

We assume the following induction hypothesis
\begin{equation} \label{Induction-Hypo}
	1=\alpha (k'-1,k'-1) > \alpha(k',k'-1) > \cdots > \alpha (m,k'-1) >0.
\end{equation}

Then, for any integer  $k \in \{k',k'+1,\ldots ,m\}$, 
	we can show that
\begin{eqnarray*}
\lefteqn{
 \alpha (k, k')- \alpha (k+1,k') 
}\\
&=&
 \frac{\lambda_{k'  }^{1+k-k'}}{(1+k-k')!}
  +  \sum_{b=0}^{k-k'}
	\left(
		\frac{\lambda_{k'}^{b}}{b!} \; 
		\alpha (k-b,k'-1) 
	\right)
- \frac{\lambda_{k'  }^{2+k-k'}}{(2+k-k')!}
  -  \sum_{b=0}^{1+k-k'}
	\left(
		\frac{\lambda_{k'}^{b}}{b!} \; 
		\alpha (k+1-b,k'-1) 
	\right) \\
&=&
 \frac{\lambda_{k'  }^{1+k-k'}}{(1+k-k')!}
-\frac{\lambda_{k'  }^{2+k-k'}}{(2+k-k')!}
-\frac{\lambda_{k'  }^{1+k-k'}}{(1+k-k')!}\alpha (k',k'-1)\\
&&
+\sum_{b=0}^{k-k'}
	\left(
		\frac{\lambda_{k'}^{b}}{b!} \; 
		\Bigl(
			\alpha (k-b,k'-1) - \alpha (k+1-b,k'-1)
		\Bigr)
	\right) \\
&\geq &
 \frac{\lambda_{k'  }^{1+k-k'}}{(1+k-k')!} 
	\Bigl( 1- \alpha (k',k'-1) \Bigr)
-\frac{\lambda_{k'  }^{2+k-k'}}{(2+k-k')!} 
\mbox{\hspace{10ex}(obtained from~(\ref{Induction-Hypo}))}   \\
&=&
 \frac{\lambda_{k'  }^{1+k-k'}}{(1+k-k')!} \lambda_{k'  }
-\frac{\lambda_{k'  }^{2+k-k'}}{(2+k-k')!} 
 \mbox{\hspace{26ex}(obtained from~(\ref{lambda_K-induction}))} \\
&=&  \frac{\lambda_{k'  }^{2+k-k'}}{(1+k-k')!}
	\left(
		1- \frac{1}{2+k-k'}
	\right)
	> 0.  \\[-5ex]
\end{eqnarray*}
\hfill \qed

\bigskip
Now we describe a proof of Theorem~\ref{PositiveSol}.

\underline{\mbox{Proof of Theorem~\ref{PositiveSol}.}}
Obviously, $\lambda_1 =1 >0$.
Equalities~(\ref{lambda_K-induction}) and
	Lemma~\ref{alpha} imply that
\begin{eqnarray*}
	\forall k' \in \{2,3, \ldots , m\},&& \;\;\; 
	\lambda_{k'}=1- \alpha(k',k'-1)
	=\alpha (k'-1,k'-1)-  \alpha(k',k'-1)
	>0.  \\[-5ex]
\end{eqnarray*}
 \hfill \qed

Recurrence relations (\ref{alpha_K-induction}) and (\ref{lambda_K-induction})
	give the following efficient algorithm for calculating 
	$(\lambda_1,\lambda_2, \ldots, \lambda_m)$ without enumerating
	vectors in $\Xi_k$.

\begin{description}
\setlength{\leftskip}{1cm}
\setlength{\itemindent}{-0.5cm}
\item[\underline{\rm Algorithm A}]
\item[\rm Step~0:]
	Set $k':=1$; $\lambda_1:=1$; 
	$\alpha (k,1):=1/k! \,$ for all $k \in \{1,2,\ldots,m\}$.
\item[\rm Step~1:]
	Set $k':=k'+1$; $\alpha (k',k'):=1$;  $\lambda_{k'}:=1-\alpha (k',k'-1)$. \\
	For each $k \in \{k'+1,\ldots,m\}$, 
	calculate $\alpha (k,k')$ by (\ref{alpha_K-induction}).
\item[\rm Step~2:]
	If $k'=m$, then stop and output $(\lambda_1,\lambda_2,\ldots, \lambda_m)$.
	Else, goto Step~1.
\end{description}

\medskip
\noindent
The total number of basic arithmetic operations 
	required in Algorithm~A
	is bounded by $\Order (m^3)$.


\section*{Acknowledgments.}

 The authors thank Professor A\@.~V\@.~Gnedin 
 for his comment on using a Poisson approximation 
 to find the asymptotic probability of win 
 at the 34th Conference on Stochastic Processes and their Applications, Osaka.

This work was supported by JSPS KAKENHI Grant Numbers 26285045, 26242027.
 

\nocite{CRS}
\nocite{ASSAF2000}
\nocite{BRUSS2009}
\nocite{FERGUSON2006}
\nocite{FERGUSON2008}
\nocite{GNEDIN2010}
\nocite{SHIRYAEV1978}



\end{document}